\documentclass[12pt]{article}

\usepackage{graphicx}
\usepackage{latexsym,amssymb}
\usepackage{amsthm}
\usepackage{indentfirst}
\usepackage{amsmath}
\usepackage{color}
\usepackage{xcolor}
\usepackage{fourier}
\usepackage[colorlinks=true,backref=page]{hyperref}
\usepackage[colorlinks=true]{hyperref}
\usepackage[all]{xy}

\usepackage[a4paper,text={160true mm,230true mm},centering,top=35true mm,head=5true mm,headsep=2.5true mm,foot=8.5true mm]{geometry}

\newtheorem{theoremalph}{Theorem}

\newtheorem*{Main Theorem}{Main Theorem}

\newtheorem{Theorem}{Theorem}[section]
\newtheorem*{Theorem A}{Theorem A}
\newtheorem*{Theorem A'}{Theorem A'}
\newtheorem*{Theorem B'}{Theorem B'}

\newtheorem{Definition}[Theorem]{Definition}
\newtheorem{Proposition}[Theorem]{Proposition}
\newtheorem{Lemma}[Theorem]{Lemma}

\newtheorem{Remark}[Theorem]{Remark}

\newtheorem{Remark-numbered}[Theorem]{Remark}
\newtheorem{Corollary}[Theorem]{Corollary}

\newtheorem*{Claim}{Claim}
\newtheorem{Claim-numbered}[Theorem]{Claim}

 \def\NN{{\mathbb N}} 

 \def\RR{{\mathbb R}}

 \def\ZZ{{\mathbb Z}}

\def\cB{{\cal B}}    
    
   \def\cP{{\cal P}}

\newcommand{\loc}{\operatorname{loc}}

\def\dim{\operatorname{dim}}

\def\Jac{\operatorname{Jac}}

\def\Leb{\operatorname{Leb}}

\def\PES{\operatorname{PES}}

\begin{document}

\title{The Bernoulli property of Sinai-Ruelle-Bowen measures for flows}

\author{Chiyi Luo and Dawei Yang\footnote{
D. Yang  was partially supported by National Key R\&D Program of China (2022YFA1005801),  NSFC 12325106, NSFC12526207, ZXL2024386 and Jiangsu Specially Appointed Professorship.
C. Luo is partially supported by NSFC 12501244 and Early-Career Young Scientists and Technologists Project of Jiangxi Province.
}}

\date{}
\maketitle

\begin{abstract} 
We prove that for $C^{1+\beta}$ flows whose generating vector fields may have singularities, every weakly mixing hyperbolic SRB measure is Bernoullian. 
\end{abstract}

\tableofcontents

\section{Introduction}\label{SEC:1}
The Bernoulli property has attracted considerable attention. 
In the setting of smooth dynamical systems, the Bernoulli property for linear hyperbolic toral automorphisms was established by Ornstein and Weiss~\cite{OrW73}. 
For general equilibrium states of uniformly hyperbolic diffeomorphisms, the Bernoulli property was proved by Bowen~\cite{Bow74}.
For uniformly hyperbolic flows, it was first obtained for geodesic flows of constant negative curvature~\cite{OrW73}, and later extended to general Anosov flows by Ratner~\cite{Rat74}.
We also note that Anosov~\cite{Ano67} made a fundamental contribution by establishing the K-property.

Much research interest has shifted to the theory of non-uniform hyperbolicity after the work of Pesin \cite{Pes76,Pes77,Pes77G}. 
In \cite{Pes77}, Pesin proved that some ergodic component of the Liouville measure is Bernoullian for geodesic flow with non-positive and non-identically zero curvature.
This result can be adapted to volume measures associated with non-uniformly hyperbolic systems.
Furthermore, for $C^{1+}$ diffeomorphisms, this theory was extended by Ledrappier~\cite{Led84} to hyperbolic weakly mixing measures whose conditional measures on unstable manifolds are absolutely continuous with respect to the corresponding Lebesgue measures.
The Bernoulli property also has been extended to several equilibrium states for geodesic flows on rank-1 manifolds; see, for instance, \cite{BD22,BCFT18}.

For general equilibrium states of Hölder continuous potentials or geometric potentials, Sarig~\cite{Sar11,Sar13} developed a framework, often referred to as a ``coding'', to establish the Bernoulli property for diffeomorphisms.
In higher dimensions, when the equilibrium state is hyperbolic, the coding constructed by Ben Ovadia~\cite{BOv18} is expected to be useful under suitable conditions.
Sarig's framework can also be applied to flows generated by vector fields under the additional assumption that the vector field has no singularities, or equivalently, that the flow has no fixed points; see, for instance, \cite{LLS16,LiS19}.
In the higher-dimensional setting for no singular flows, the coding approach developed in \cite{LMN25} is expected to be helpful.
The

The aim of this paper is to study general flows generated by vector fields, without imposing the assumption that the vector field has no singularities, or equivalently, that the flow has positive speed.
For general equilibrium states, establishing the Bernoulli property remains out of reach at present, due to the lack of an appropriate coding.
In this paper, we restrict our attention to SRB measures, which are equilibrium states for geometric potentials.
For diffeomorphisms, this is one of the main results of Ledrappier~\cite{Led84}.
For flows without singularities, it can be obtained in~\cite{ALP24,LLS16,LiS19} by using coding.
The case of general flows has not been previously considered\footnote{
The work of Ornstein-Weiss~\cite{OrW98} appears to be closely related to the main result of this paper.
However, their approach requires the existence of a suitable cross-section and a regular roof function.
Chernov-Haskell~\cite{ChH96} studied non-uniformly hyperbolic systems, but the notion of singularities in their setting differs from ours, and their three conditions on singularities may not be satisfied in the present context.}.

Assume that $M$ is a compact boundaryless Riemannian manifold and that $\{T^t\}_{t\in\mathbb R}$ is a flow generated by a $C^{1+\beta}$ vector field\footnote{In this paper, we will rarely refer explicitly to the vector field, and hence we do not introduce separate notation for it.}.
In this case, the flow $\{T^t\}_{t\in\mathbb R}$ is of class $C^{1+\beta}$.
When there is no ambiguity, we abbreviate the flow by $T^t$.
The singularities of the vector field correspond to the fixed points of the flow $T^t$.
Denote
$${\rm Fix}(T^t)=\{p\in M:~T^t(p)=p,~\forall t\in\mathbb R\}.$$
An invariant measure $\mu$ is said to be \emph{regular} if $\mu({\rm Fix}(T^t))=0$.

Flows with fixed points (or, equivalently, vector fields with singularities) are of particular significance. 
The famous Lorenz attractor \cite{Lor63} is one of these examples. 
From a purely mathematical point of view, there are “geometric Lorenz attractors” constructed in~\cite{ABS77,Guc76,GuW79}.
To describe Lorenz-like dynamics in a broader setting, \cite{MPP04} introduced the notion of ``singular hyperbolicity''. 
We refer to~\cite{APPV09,CYZ20,LeY17} for results on the existence and uniqueness of SRB measures for singular hyperbolic attractors.
We also note that deeper ergodic properties of SRB measures for general singular hyperbolic attractors remain an interesting direction for further study.

From Oseledec~\cite{Ose68}, for an invariant measure $\mu$, $\mu$-almost every point has $\dim M$ well-defined Lyapunov exponents.
An invariant regular measure is said to be \emph{hyperbolic} if it has exactly one zero Lyapunov exponent.
Note that for ergodic measures supported on ${\rm Fix}(T^t)$, it may happen that all Lyapunov exponents are nonzero.
For more discussion on the hyperbolicity of measures, one can see Section~\ref{Sec:Pesin}.

SRB measures are invariant measures with nice geometric properties.
An invariant measure $\mu$ of a $C^{1+\beta}$ flow $T^t$ is said to be a Sinai-Ruelle-Bowen measure (SRB measure for short) if it has positive Lyapunov exponents, and its conditional measures on the unstable manifolds are absolutely continuous with respect to the Lebesgue measures on these submanifolds.
SRB measures have positive metric entropy by~\cite{LeY85}.
Thus, ergodic SRB measures must be regular.
Ergodic SRB measures need not be Bernoullian, since the flow may be the suspension of a hyperbolic diffeomorphism with constant roof function.
We need some additional weak property to exclude this case.
From \cite[Chapter 3.4]{FiH19}, an invariant measure $\mu$ is said to be \emph{weakly mixing} if for any measurable set $A,B$, one has that
$$\lim_{t\to \infty}\frac{1}{t} \int_0^t |\mu(A \cap T^{-s}(B)) - \mu(A)\mu(B)| \, {\rm d}s = 0.$$

We first recall the definition of Bernoulli property. 
An automorphism $(X,\mathcal B,\mu,T)$ is said to be Bernoullian if it is isomorphic to a Bernoulli scheme\footnote{One can see the definition of a Bernoulli scheme in~\cite[Section 3.3]{Sar11}. The definition in~\cite{Sar11} differs from that in~\cite[Page 105]{Wal82}. However, by Ornstein's theorem~\cite{Orn74}, metric entropy is a complete invariant of Bernoulli systems, and hence the two definitions are equivalent.}.

\begin{Definition}
Let $(X,\mathcal B,\mu)$ be a probability space.
A flow $(X,\mathcal B,\mu,{T^t})$ is said to be \emph{Bernoullian} if the time-one map $T=T^1$ defines a Bernoulli automorphism, that is, if $(X,\mathcal B,\mu,T)$ is Bernoullian.
\end{Definition}
The Bernoulli property is a strong ergodic property; it implies the K-property, which in turn implies mixing.
Our main result shows that weakly mixing hyperbolic SRB measures for flows (with singularities) are Bernoullian.

\begin{theoremalph}\label{Thm:main}
Assume that $\mu$ is a hyperbolic SRB measure for a $C^{1+\beta}$ flow $\{T^t\}_{t\in\mathbb R}$.
If $\mu$ is weakly mixing, then $\mu$ is Bernoullian.
\end{theoremalph}

For the proof of Theorem~\ref{Thm:main}, in contrast to~\cite{Led84}, we must take into account the specific features of the flow.
See, for instance, Section~\ref{Sec:to-prove-K} for details.

To prove that an ergodic measure is Bernoullian, there is a classical approach due to Ornstein and Weiss~\cite{OrW73}: one first establishes the K-property and then constructs a sequence of very weak Bernoulli partitions generating the $\sigma$-algebra.
However, different settings require different strategies to implement these steps.
In the present setting of hyperbolic SRB measures for singular flows, no coding is available as in~\cite{LiS19}.
Thus, we do not have a local product structure arising from symbolic dynamics.
Moreover, it is not clear how to adapt the method of Ledrappier~\cite{Led84} to flows with fixed points.
Even in the absence of fixed points, such as for geodesic flows, additional arguments are typically required; see Anosov~\cite{Ano67} and Section~\ref{Sec:to-prove-K}, for instance.

We now recall the K-property for measure-preserving systems.
Note that the K-property implies K-mixing, which in turn implies mixing.  
We refer to~\cite[Chapter 10.2]{CFS82} for a detailed description.
\begin{Definition}[Tail $\sigma$-algebra and K-property] \label{Def:Kp}
Let $(X,\mathcal B,\mu,T)$ be a measure-preserving system, and let $\mathcal A\subset \mathcal B$ be a $\sigma$-sub-algebra satisfying $T^{-1}\mathcal A\subset \mathcal A$. 
The tail $\sigma$-algebra of $\mathcal A$ is defined by
$${\rm Tail}(\mathcal A)=\bigcap_{n\ge 0}T^{-n}\mathcal A.$$
The system $(X,\mathcal B,\mu,T)$ is said to have the K-property if there exists a $\sigma$-subalgebra $\mathcal A\subset \mathcal B$ with $T^{-1}\mathcal A\subset \mathcal A$ such that
\begin{itemize}
\item ${\rm Tail}(\mathcal A)=\{\emptyset,X\}$ modulo $\mu$;
\item $\bigvee_{n=0}^\infty T^n \mathcal A=\mathcal B$ modulo $\mu$.
\end{itemize}  
\end{Definition}
Note that, in the proof, we will consider $\mathcal A$ to be generated by partitions subordinated to stable manifolds.
Theorem~\ref{Thm:main} can be reduced to the following two theorems.
\begin{Theorem}\label{Thm:weakly-to-K}
Assume that $\mu$ is a hyperbolic SRB measure of a flow $\{T^t\}_{t\in\mathbb R}$. If $\mu$ is weakly mixing, then $\mu$ has the K-property.
\end{Theorem}
\begin{Theorem}\label{Thm:K-to-B}
Assume that $\mu$ is a hyperbolic SRB measure of a flow $\{T^t\}_{t\in\mathbb R}$. If $\mu$  has the K-property, then $\mu$ is Bernoullian.
\end{Theorem}

\section{Pesin theory}\label{Sec:Pesin}
In this section, we recall Pesin’s non-uniform hyperbolic theory. 
For a hyperbolic measure, Pesin blocks can be defined via the time-one map, and the Bernoulli property of the flow follows from that of the time-one map.
The proof of the Bernoulli property mainly relies on the properties of the local stable and unstable foliations.
Therefore, singularities of the flow do not give rise to significant difficulties in this local setting.

\subsection{Pesin Blocks}
Let $\{T^t\}_{t\in\mathbb R}$ be a $C^{1+\beta}$ flow.
An invariant measure $\mu$ is called \textit{regular}, if $\mu$ gives zero measure to the set of singularities. 
Let $\mu$ be an ergodic regular hyperbolic measure of the flow $\{T^t\}_{t\in\mathbb R}$. 
Then, there exist a full $\mu$-measure set $\Lambda_{\mu}$ and a $\{DT^t\}_{t\in \RR}$-invariant measurable splitting $T_{\Lambda_{\mu}}M=E^u\oplus E^c\oplus E^s$ such that for every $x\in \Lambda_{\mu}$,
\begin{itemize}
\item $E^u(x)$ is the subspace corresponding to positive Lyapunov exponents;
\item $E^c(x)$ is the subspace generated by the flow direction (in particular, $\dim E^c=1$);
\item $E^s(x)$ is the subspace corresponding to negative Lyapunov exponents.
\end{itemize}

Given $0<\varepsilon_0\ll\chi$ and $\ell\in\mathbb N$, define the Pesin block ${\rm PES}_{\ell}^{\chi,\varepsilon_0}(T^t)$ as follows: $x\in {\rm PES}_{\ell}^{\chi,\varepsilon_0}(T^t)$ if for every $n\in\mathbb N$ and $k\in\mathbb Z$
\begin{itemize}
	\item $\|D_{T^k(x)}T^{-n}|_{E^u(T^k(x))}\|\leq \ell {\rm e}^{|k|\varepsilon_0}{\rm e}^{-(\chi-\varepsilon_0)n}, \ \|D_{T^k(x)}T^{n}|_{E^s(T^k(x))}\|\leq \ell {\rm e}^{|k|\varepsilon_0} {\rm e}^{-(\chi-\varepsilon_0)n}$;
	\item ${\rm m}(D_{T^k(x)}T^{n}|_{E^u(T^k(x))})\geq  \ell^{-1}{\rm e}^{-|k|\varepsilon_0}{\rm e}^{(\chi-\varepsilon_0)n}, \ {\rm m}(D_{T^k(x)}T^{-n}|_{E^s(T^k(x))})\geq \ell^{-1}{\rm e}^{-|k|\varepsilon_0}{\rm e}^{(\chi-\varepsilon_0)n}$;
	\item $\angle (E^{\tau_1} (T^k(x)), E^{\tau_2} (T^k(x)))\geq {\ell}^{-1}{\rm e}^{-k\varepsilon_0},~\tau_1\neq \tau_2\in \{u,c,s\}$;
	\item $\ell^{-1}{\rm e}^{-(|k|+n)\varepsilon_0}\le  \|D_{T^k(x)}T^{n}|_{E^c(T^k(x))}\|\leq \ell {\rm e}^{(|k|+n)\varepsilon_0}$;
\end{itemize}
where ${\rm m}(A):=\inf\{\|Av\|:\|v\|=1\}$ is the co-norm for linear maps.
If there is no confusion, we denote ${\rm PES}_{\ell}^{\chi,\varepsilon_0}(T^t):={\rm PES}_{\ell}^{\chi,\varepsilon_0}$.

Fix $\chi>0$. 
An ergodic hyperbolic measure $\mu$ is said to be \emph{$\chi$-hyperbolic} if all Lyapunov exponents in $E^u$ are larger than $\chi$ and all Lyapunov exponents in $E^s$ are small than $-\chi$.
It is clear that if $\mu$ is an ergodic regular hyperbolic measure, then there exists $\chi>0$ such that for any $0<\varepsilon_0\ll \chi$, one has $\mu(\bigcup_{\ell>0}{\rm PES}_{\ell}^{\chi,\varepsilon_0})=1$.

\subsection{Local unstable manifolds}

From Pesin theory, every point $x\in {\rm PES}_{\ell}^{\chi,\varepsilon_0}$ has a local unstable manifold $W^u_{\rm loc}(x)$. 
For instance, can see \cite[Section 2]{LuY24} for the precise definition.

\begin{Theorem}[Pesin stable/usntable manifold theorem]\label{Thm:Pesin-manifold}
Given $\chi>0$, for $0<\varepsilon_0\ll\chi$ and $\ell\in\mathbb N$, there is $\delta=\delta(\chi,\varepsilon_0,\ell)>0$ such that every point $x\in {\rm PES}_{\ell}^{\chi,\varepsilon_0}$ has a local unstable manifold $W^u_\delta(x)$ in the form 
$$W^u_\delta(x)=\exp_{x}\big\{(v,\psi^u_x(v)): v\in E^u(x)~\text{with}~\|v\| \leq \delta \big\},$$
where $\psi^u_x: \{v\in E^{u}(x):\|v\|\leq  \delta\} \to E^{cs}(x)$ is a $C^{1+\beta}$ map satisfying $\psi^u_x(0)=0,~D_0(\psi^u_x)=0$ and ${\rm Lip}(\psi^u_x)\leq 1/10$.
Moreover, for every point $x\in {\rm PES}_{\ell}^{\chi,\varepsilon_0}$ and all points $y,z\in W^u_{\delta}(x)$, one has that
$$d(T^{-t}(y),T^{-t}(z))\le 2\ell {\rm e}^{-t(\chi-\varepsilon_0)}d(y,z),~~~\forall t\ge 0.$$
The same result holds for local stable manifolds.
And $W^{u}_{\rm loc}(x),W^{s}_{\rm loc}(x)$ depend continuously on $x\in \PES_{\ell}^{\chi,\varepsilon_0}$ in the $C^1$ topology.
\end{Theorem}

\begin{Remark}\label{Rem:reduce-size}
The quantity $\delta(\chi,\varepsilon_0,\ell)$ satisfies $\delta(\chi,\varepsilon_0,\ell+1)\geq e^{-\varepsilon_0}\delta(\chi,\varepsilon_0,\ell)$.
When $\chi$ and $\varepsilon_0$ are fixed, we write $\delta(\chi,\varepsilon_0,\ell)$ simply as $\delta(\ell)$.
We may reduce the size $\delta$ in Theorem~\ref{Thm:Pesin-manifold} and use a much smaller value. 
Hence, we denote by $W^{u/s}_{\rm loc}(x)$ the local stable/unstable manifolds with the maximal size allowed by Theorem~\ref{Thm:Pesin-manifold}.
\end{Remark}

For each $x\in \bigcup_{\ell>0}{\rm PES}_{\ell}^{\chi,\varepsilon_0}$, the global stable and unstable manifolds are defined by
$$ W^u(x):=\bigcup_{t> 0} T^{t}(W^u_{\loc}(T^{-t}(x))),~~W^s(x):=\bigcup_{t> 0} T^{-t} (W^s_{\loc}(T^{t}(x))).$$
For flows, one also has the center-unstable and center-stable foliations $W^{cu}$ and $W^{cs}$, whose leaves are defined as follows: for $\mu$-almost every $x$,
$$W^{cu}(x)=\bigcup_{t\in\mathbb R}T^t(W^u(x)),~~~W^{cs}(x)=\bigcup_{t\in\mathbb R}T^t(W^s(x)).$$
One can also define the local center-stable and center-unstable manifolds. 
Recall that $\delta=\delta(\chi,\varepsilon_0,\ell)$ is the size of the local stable and unstable manifolds given by Theorem~\ref{Thm:Pesin-manifold}. 
For $x\in {\rm PES}_{\ell}^{\chi,\varepsilon_0}$, we define
$$W^{cu}_{\delta}(x)=T^{[-\delta,\delta]}(W^u_\delta(x)),~~W^{cs}_{\delta}(x)=W^{cs}_{\rm loc}(x)=T^{[-\delta,\delta]}(W^s_\delta(x)).$$
These embedded manifolds can also be represented as $C^{1+\beta}$ graphs with small Lipschitz constants.
We also denoted it by $W^{cu/cs}_{\rm loc}(x)$ when we do not wish to emphasize the precise value of $\delta$, referring to the maximal size allowed.
We recall the local product structure on Pesin blocks.

\begin{Lemma}\label{Lem:transverse-intersection}
Under the assumptions of Theorem~\ref{Thm:Pesin-manifold}, given $\ell \in \NN$, there exist $L\in\mathbb N$ and a constant $0<\varepsilon(\ell)<1$ with the following property: There exists $\alpha_0>0$ such that for every $0<\alpha<\alpha_0$ and every $x,y\in {\rm PES}_{\ell}^{\chi,\varepsilon_0}$ with $d(x,y)<\alpha\varepsilon(\ell)$, one has that $W^{\tau}(y)\cap B(x,\alpha)$ is connected for $\tau\in \{u,s,cu,cs\}$, and 
\begin{align*}
	W^{s}_{\loc}(x) \pitchfork W^{cu}_{\loc}(y) =\{z_1\}\subset {\rm PES}^{\chi,\varepsilon_0}_{L}\cap B(x,\alpha),~~
	W^{u}_{\loc}(x) \pitchfork W^{cs}_{\loc}(y) =\{z_2\}\subset {\rm PES}^{\chi,\varepsilon_0}_{L}\cap B(x,\alpha).
\end{align*}
\end{Lemma}

\subsection{Rectangles}
\begin{Definition}\label{Def:rectangle}
	For $\alpha>0$, a measurable set $\Pi^s$ is said to be an $(\alpha,s)$-rectangle at a point $w\in {\rm PES}_{\ell}$ if the following properties hold:
	\begin{enumerate}
		\item $w\in\Pi^s\subset {\rm PES}_{L}^{\chi,\varepsilon_0}\cap B(w,\alpha)$, $L$ be the integer corresponding to $\ell$ given by Lemma~\ref{Lem:transverse-intersection};
		\item $W^{cu}_{\rm loc}(y)\cap W^s_{\rm loc}(z)\in\Pi,~\forall y,z\in \Pi^s$. 
	\end{enumerate}
	Similarly, one can define  $(\alpha,u)$-rectangle.
	When $\alpha$ is fixed, they are simply called $s$-rectangles and $u$-rectangles.
\end{Definition}

\begin{Proposition}\label{Pro:rectangle-covering}
	Let $\ell \in \NN$ and let $L$ be the integer corresponding to $\ell$ given by Lemma~\ref{Lem:transverse-intersection}.
	For every $\alpha>0$, there are disjoint $(\alpha,s)$-rectangles $\{\Pi_1^s,\Pi_2^s,\cdots,\Pi_{N_*}^s\}$ such that
	$${\rm PES}_{\ell}\subset \bigcup_{i=1}^{N_*} \Pi_{i}^s\subset {\rm PES}_{L}.$$
	Similarly, the same conclusion holds for $(\alpha,u)$-rectangles.
\end{Proposition}
\begin{proof}
    For every $\alpha>0$, we first show that there are finitely many $(\alpha,s)$-rectangles $\{\Pi_1^s,\cdots,\Pi_{N}^s\}$ such that ${\rm PES}_{\ell} \subset \bigcup_{j=1}^{N} \Pi_j^s \subset {\rm PES}_{L}$.
    
    By the compactness of ${\rm PES}_{\ell}$, there exist finitely many points $\{w_i\}_{i=1}^{N^*} \subset {\rm PES}_{\ell}$ such that ${\rm PES}_{\ell}= \bigcup_{j=1}^{N}B(w_j,\alpha \varepsilon(\ell)) \cap {\rm PES}_{\ell}$, where $\varepsilon(\ell)$ is given by Lemma~\ref{Lem:transverse-intersection}.
    For each $1\leq j\leq N$,
    $$\Pi_{j}^s:=\big\{y:~ \exists x,z\in {\rm PES}_{\ell}\cap B(w_j,\alpha\varepsilon(\ell)),~\text{s.t.}~W^{s}_{\rm loc}(x) \pitchfork W^{cu}_{\rm loc}(z)=\{y\}\big\}.$$
    It remains to verify that $\Pi_j^s$ is an $s$-rectangle.
    By reducing $\alpha>0$ if necessary, given any two points $z_1,z_2\in \Pi_j^s$, there exist points $x_1,x_2,y_1,y_2\in {\rm PES}_{\ell}\cap B(w_j,\alpha\varepsilon(\ell))$ such that
    $$z_1\in W^s_{\loc}(x_1) \pitchfork W^{cu}_{\loc}(x_2),~z_2\in W^s_{\loc}(y_1) \pitchfork W^{cu}_{\loc}(y_2).$$
    Therefore, 
    $$ W^s_{\loc}(z_1) \pitchfork W^{cu}_{\loc}(z_2)=W^s_{\loc}(x_1) \pitchfork W^{cu}_{\loc}(y_2)\subset \Pi_j^s.$$
    This shows that $\Pi_j^s$ is an $(\alpha,s)$-rectangle.
    However, $\{\Pi_1^s,\cdots,\Pi_{N}^s\}$ are not pairwise disjoint.
    We will need the following result:
    \begin{Claim}\label{Lem:split-rectangle}
    	Let $\Pi$ and $\Pi'$ be two $(\alpha,s)$-rectangles with $\Pi \cap \Pi' \neq \emptyset$. 
    	Then there exist at most five $(\alpha,s)$-rectangles $\{\Pi_i\}_{i=1}^{5}$ such that ${\rm PES}_{\ell}\cap(\Pi_1 \cup \Pi_2 \cup \cdots \cup \Pi_5)={\rm PES}_{\ell}\cap(\Pi\cup \Pi' )$ and $\Pi_i \cap \Pi_j \neq \emptyset$ for all $i\neq j$.
    \end{Claim}
    \begin{proof}
    	We define the five rectangles as follows:
    	\begin{align*}
    		&\Pi_1=\Pi \cap \Pi',~~\Pi_2=\Pi\setminus\bigcup_{y\in \Pi_1} W^s_{\rm loc}(y),~~\Pi_3=\Pi'\setminus \bigcup_{y\in \Pi_1} W^s_{\rm loc}(y), \\
    		&\Pi_4=(\Pi\setminus\Pi_1)\cap \bigcup_{y\in \Pi_1} W^s_{\rm loc}(y),~~\Pi_5= (\Pi'\setminus\Pi_1)\cap \bigcup_{y\in \Pi_1} W^s_{\rm loc}(y).
    	\end{align*}
    	By construction, we have $\Pi \cup \Pi'= \Pi_1 \cup \cdots \cup \Pi_5$ and $\Pi_i \cap \Pi_j=\emptyset,~i\neq j$.
    	It remains to verify that each $\Pi_i$ is an $(\alpha,s)$-rectangle. 
    	A corresponding statement for diffeomorphisms can be found in \cite[Lemma 8.2]{Pes77}.
        \begin{enumerate}
        	\item [$\Pi_1$:] For any $y,z\in \Pi_1=\Pi\cap\Pi’$, we have $W^{cu}_{\rm loc}(y)\cap W^s_{\rm loc}(z)\in\Pi\cap\Pi'=\Pi_1$. 
        	Thus, $\Pi_1$ is an $(\alpha,s)$-rectangle.
        	\item [$\Pi_2$:] For any $y,z\in \Pi_2=\Pi\setminus \bigcup_{x\in \Pi_1} W^s_{\rm loc}(x)$, we have $w\in W^{cu}_{\rm loc}(y)\cap W^s_{\rm loc}(z)\subset\Pi$.
        	If $w\in \bigcup_{x\in \Pi_1} W^s_{\rm loc}(x)$, then we also have $z\in W^s_{\rm loc}(x)$ for some $x\in \Pi_1$, which is a contradiction. 
        	Therefore, $W^{cu}_{\rm loc}(y)\cap W^s_{\rm loc}(z)\in\Pi\setminus\bigcup_{x\in \Pi_1} W^s_{\rm loc}(x)$.
        	\item [$\Pi_3$:] This is similar to the case of $\Pi_2$.	 
        	\item [$\Pi_4$:] For any $y,z\in\Pi_4= (\Pi\setminus\Pi_1)\cap \bigcup_{y\in \Pi_1} W^s_{\rm loc}(y)$, since $\Pi$ is a rectangle, we have $w\in W^{cu}_{\rm loc}(y)\cap W^s_{\rm loc}(z)\subset\Pi$. 
        	Since $z\in W_{\rm loc}^s(x)$ for some $x\in \Pi_1$, it follows that $z\in W^s_{\rm loc}(x)$. 
        	Now suppose, by contradiction, that $w\in \Pi_1$.
        	Choose $x\in \Pi_1$ such that $y\in W^s_{\rm loc}(x)$. 
        	Then $y\in W^s_{\rm loc}(x)\cap W^{cu}_{\rm loc}(w)\subset\Pi_1$ since $\Pi_1$ is a rectangle. 
        	This is a contradiction. 
        	Hence $w\notin \Pi_1$, and it follows that $\Pi_4$ is an $(\alpha,s)$-rectangle.
        	\item [$\Pi_5$:] This is similar to the case of $\Pi_4$.	
        \end{enumerate}
        This completes the proof of the Claim
    \end{proof}
        The proof of Proposition \ref{Pro:rectangle-covering} follows directly from the claim.
\end{proof}

\subsection{Holonomy maps and absolutely continuity of unstable foliations}

We recall the definition of local stable holonomy maps $h^s$ and the local center-stable holonomy maps $h^{cs}$.

\begin{Definition}
	Fix $\ell\in \mathbb{N}$ and take $\delta:=\delta(\ell)$ be as in Lemma \ref{Thm:Pesin-manifold} and $\varepsilon(\ell)$ be as in Lemma~\ref{Lem:transverse-intersection}.
	\begin{itemize}
	\item	For any points $x,y\in \PES_{\ell}$ with $d(x,y)<\delta\varepsilon(\ell)$ and any $z\in W^{cu}_{\delta/2}(x)\cap {\rm PES}_{\ell}$, the manifold $W^{cu}_{\rm loc}(y)$ intersects $W^{s}_{\rm loc}(z)$ at a unique point $w$. 
    Then the map $h_{x,y}^s$ defined by $h_{x,y}^s(z):=w$ is a map from $W^{cu}_{\delta/2}(x)\cap {\rm PES}_{\ell}^{\chi,\varepsilon_0}$ to $W^{cu}_{\rm loc}(y)$.
    \item	For any points $x,y\in \PES_{\ell}$ with $d(x,y)<\delta\varepsilon(\ell)$ and any $z\in W^{u}_{\delta/2}(x)\cap {\rm PES}_{\ell}$, the manifold $W^{u}_{\rm loc}(y)$ intersects $W^{cs}_{\rm loc}(z)$ at a unique point $w$. 
    Then the map $h_{x,y}^{cs}$ defined by $h_{x,y}^{cs}(z):=w$ is a map from $W^{u}_{\delta/2}(x)\cap {\rm PES}_{\ell}^{\chi,\varepsilon_0}$ to $W^{u}_{\rm loc}(y)$.
	\end{itemize}
	In particular, if $~\Pi$ is a $(\alpha,u)$-rectangle for some $\alpha\ll \delta$, the map $h_{x,y}^{cs}:\Pi \cap W^{u}_{\rm loc}(x)\to \Pi \cap W^{u}_{\rm loc}(y)$ is a bijection.
\end{Definition}
Both the local stable holonomy $h^s$ and the local center-stable holonomy $h^{cs}$ depend on the size $\delta$. 
For notational simplicity, we omit the explicit dependence on $\delta$.

Let $(X, \cB_X,\mu_X)$ and $(Y, \cB_Y,\mu_Y)$ be two probability space, and let $f:X\to Y$ be a measurable transformation.
A measurable function $\Jac(f): X \to [0,+\infty)$ is called a Jacobian of $f$ with respect to $\mu_X$ and $\mu_Y$ if
$$ \mu_Y(A)=\int_{f^{-1}(A)} \Jac(f)(x)~{\rm d}\mu_X(x),~\forall A\in \cB_Y.$$
Note that such a Jacobian exists if and only if $\mu_Y$ is absolutely continuous with respect to the pushforward measure $\mu_X \circ f^{-1}$.

Denote by ${\rm Leb}_{W^u(x)}$ and ${\rm Leb}_{W^{cu}(x)}$ the Lebesgue measures on $W^u(x)$ and $W^{cu}(x)$ induced by the Riemannian metric, respectively.
We now recall the absolute continuity of the (center) unstable foliation.
All Jacobians are understood to be taken with respect to the Lebesgue measures on the corresponding manifolds.
\begin{Theorem}\label{Thm:holonomy}
Fix $\ell\in \NN$ and $\delta:=\delta(\ell)$ be as in Lemma \ref{Thm:Pesin-manifold}.
For each $\varepsilon>0$, there exists $\alpha>0$ such that, for every $x,y\in \PES_{\ell}^{\chi,\varepsilon_0}$ with $d(x,y)<\alpha$,  $h^s_{x,y}$ and $h^{cs}_{x,y}$ well defined on $W^{cu}_{\delta/2}(x)$ and $W^u_{\delta/2}(x)$, respectively, and their Jacobians satisfy 
$${\rm Jac}(h^s_{x,y})\in(1-\varepsilon,1+\varepsilon),~~~{\rm Jac}(h^{cs}_{x,y})\in(1-\varepsilon,1+\varepsilon).$$
\end{Theorem}

The estimate for ${\rm Jac}(h^s)$ in Theorem~\ref{Thm:holonomy} follows from \cite[Theorem 8.16]{BaP13}.
The corresponding property for ${\rm Jac}(h^{cs})$ holds since $h^{cs}$ can be written as the composition of $h^s$ with the flow over a short time.

\subsection{Unstable stacks and SRB properties}
Let $\mu$ be an ergodic hyperbolic measure of the flow so that $\mu(\bigcup_{\ell>0} \PES_{\ell})=1$.
Fix $\ell\in \mathbb{N}$ and take $\delta:=\delta(\ell)$ be as in Lemma \ref{Thm:Pesin-manifold}.
Let $0<r \ll \delta$ and let $E\subset B(w,r\varepsilon(\ell))\cap \PES_{\ell}$ be a compact subset of $\PES_{\ell}$ with $\mu(E)>0$.
The $r$-local unstable stack over $E$ is defined by
$$S:=S_{E}^r=\bigcup_{x\in E}W^u_{\loc}(x) \cap B(w,r).$$
Similarly, one defines the $r$-local center-unstable stacks over $E$.
For each $y\in S$, choose $x_y\in E$ and define $\xi(y):=W_{\loc}^u(x_y)\cap B(w,r)$.
It is clear that when $r$ is small enough, $\xi(y)$ does not depend on the choice of $x_y$.
Since the local unstable manifolds depend continuously on $\PES_\ell$, it follows that $\xi$ is a measurable partition of $S$.
Recall that in \cite[Proposition 3.1]{LeS82}, by choosing a suitable $r$ (there are uncountably many such choices), the partition $\eta:=\bigvee_{n\geq 0}f^{n}\hat{\xi}$,
(where $\hat{\xi}:=\xi\cup (M \setminus S)$), is a measurable partition of $M$ that subordinated to $W^u$, meaning that:
\begin{itemize}
	\item $\eta \le T^{-1}\eta$; i.e., $(T^{-1}\eta)(x)\subset \eta(x)$ for $\mu$-almost every $x$;
	\item for $\mu$-almost every $x$, $\eta(x)\subset W^u(x)$ and $\eta(x)$ contains a neighborhood of $x$ in $W^u(x)$;
	\item the partitions $\{T^n\xi^s\}_{n\leq 0}$ generate $\cB_M$ modulo $\mu$;
	\item for $\mu$-almost every $x$, $\bigcup_{n>0}T^n\big(\eta(f^{-n}(x))\big)=W^u(x)$.
\end{itemize}

Let $\mu|_{S}(\cdot):=\mu(S\cap \cdot)$. 
By the Rokhlin disintegration theorem, there exists a family of probability measures $\{\mu_{\xi(x)}\}$ and $\{\mu_{\eta(x)}\}$ such that $\mu|_{S}=\int_{S} \mu_{\xi(x)} {\rm d}\mu(x)$ and $\mu=\int \mu_{\eta(x)} {\rm d}\mu(x)$,
where $\mu_{\xi(x)}$ is supported on $\xi(x)$ and $\mu_{\eta(x)}$ is supported on $\eta(x)$.
SRB measures have nice geometric properties. 
Assume that $\mu$ is a hyperbolic SRB measure. 
Then the following theorem holds (see \cite[Section 6]{LeY85}):

\begin{Theorem}\label{Thm:SRB}
Assume that $\mu$ is an ergodic hyperbolic SRB measure, and let $\eta$ be the partition defined above.
Then, for $\mu$-amlost every $x$ we have $\mu_{\eta(x)}$ is absolutely continuous with respect to $\Leb_{W^u(x)}$, and its density function $\rho_{\eta}:~\eta(x)\to \mathbb (0,+\infty)$ can be written as
\begin{equation}
	\rho_{\eta}(y):= \frac{{\rm d}\mu_{\eta(x)}}{{\rm d}\Leb_{W^u(x)}}(y)=\frac{\Delta(x,y)}{\int_{\eta(x)}\Delta(x,y)~{\rm d}\Leb_{W^u(x)}(y)},~~\forall y\in \eta(x),
\end{equation}
where
$$\Delta(x,y):=\lim_{n \to \infty} \prod_{k=1}^{n} \frac{{\rm Jac}(D_xT^{-k}|_{E^u(x)})}{{\rm Jac}(D_yT^{-k}|_{T_yW^u(x)})}.$$
\end{Theorem}
The same result applies to $\mu_{\xi(x)}$, since $\xi(x)$ is an at most countable union of sets $\eta(y)$, $y\in \xi(x)$, each having positive $\mu_{\xi(x)}$-measure.
\begin{Corollary}\label{Cor:SRb}
	Assume that $\mu$ is an ergodic hyperbolic SRB measure, and let $\xi$ be the partition defined above.
	Then, for $\mu$-amlost every $x\in S$, the measure $\mu_{\xi(x)}$ is absolutely continuous with respect to ${\rm Leb}_{W^u(x)}$ and its density function $\rho_{\xi}:~\eta(x)\to \mathbb (0,+\infty)$ can be written as
	\begin{equation*}
		\rho_{\xi}(y):= \frac{{\rm d}\mu_{\xi(x)}}{{\rm d}\Leb_{W^u(x)}}(y)=\frac{\Delta(x,y)}{\int_{\xi(x)} \Delta(x,y)~{\rm d}\Leb_{W^u(x)}(y)},~~\forall y\in \xi(x),
	\end{equation*}
\end{Corollary}
Since $W^u_{\rm loc}$ has the bounded distortion property on a given Pesin set, there exists a constant $D(\ell)>0$ such that for every $x\in \PES_{\ell}$,
\begin{equation}\label{e.bounded-distortion}
	e^{-D(\ell)d(y,z)^{\beta}}\leq \lim_{n \to \infty} \prod_{k=1}^{n} \frac{{\rm Jac}(D_xT^{-k}|_{T_{z}W^u_{\loc}(x)})}{{\rm Jac}(D_yT^{-k}|_{T_{y}W^u_{\loc}(x)})}\le e^{D(\ell)d(y,z)^{\beta}},~~\forall y,z\in W^u_{\loc}(x).
\end{equation}
Let $\Leb_{\xi(x)}(\cdot):=\Leb_{W^u(x)}(\xi(x)\cap \cdot )/\Leb_{W^u(x)}(\xi(x))$. We have the following corollary:
\begin{Corollary}\label{Cor:SRB-U}
Under the notation of Corollary~\ref{Cor:SRb}, for every $\varepsilon>0$, by choosing $r:=r(\varepsilon,\ell)$ sufficiently small, we have for $\mu$-almost every $x\in S$ that
	\begin{equation*}
		1-\varepsilon \leq  \frac{{\rm d}\mu_{\xi(x)}}{{\rm d}\Leb_{\xi(x)}}(y)\leq 1+\varepsilon,~~\forall y\in \xi(x).
	\end{equation*}
\end{Corollary}
Note that $\xi(x)\subset W^u_{\rm loc}(x)\cap B(w,r)$. 
The proof of the corollary follows directly from \eqref{e.bounded-distortion} and Corollary~\ref{Cor:SRb}.
For the local center unstable stack 
$$S^{cu}:=\bigcup_{x\in E}W^{cu}_{\loc}(x) \cap B(w,r).$$
we define a measurable partition $\xi^{cu}$ of $S^{cu}$ by $\xi^{cu}(y):= W^{cu}_{\loc}(x_y) \cap B(w,r\alpha)$, where $x_y\in E$ is such that $y\in W^{cu}_{\rm loc}(x_y)$.
The restricted measure $\mu|_{S^{cu}}(\cdot):=\mu(S^{cu}\cap \cdot)$ admits a disintegration $\mu|_{S^{cu}}:=\int_{S^{cu}} \mu_{\xi^{cu}(x)}~{\rm d}\mu(x)$. 

Assume that $\mu$ is an ergodic hyperbolic SRB measure.
Since $W^{cu}_{\loc}$ is obtained from $W^u_{\loc}$ by flowing for a small time, by decreasing  $r>0$ if necessary, we have
\begin{equation*}
	1-\varepsilon \leq  \frac{{\rm d}\mu_{\xi^{cu}(x)}}{{\rm d}\Leb_{\xi^{cu}(x)}}(y)\leq 1+\varepsilon,~~\forall y\in \xi^{cu}(x),
\end{equation*}
where $\Leb_{\xi^{cu}(x)}(\cdot):=\Leb_{W^{cu}(x)}(\xi^{cu}(x)\cap \cdot )/\Leb_{W^{cu}(x)}(\xi^{cu}(x))$.

%
%

%
%

\section{Proof of Theorem~\ref{Thm:weakly-to-K}: the K-property}\label{Sec:to-prove-K}
 

The strategy for proving Theorem~\ref{Thm:weakly-to-K} is to show that \emph{$\mu$ is $W^s$-metrically transitive}.
This notion was originally introduced by Anosov \cite{Ano67}, who proved that the geodesic flow on a compact Riemannian manifold of negative curvature has this property.
 
\begin{Definition}
Let $\mu$ be an ergodic hyperbolic measure. 
A subset $A\subset M$ is said to be {\rm $s$-saturated (with respect to $\mu$)} if there exists a full measure set $\Gamma$ such that, for $\mu$-almost every $x\in A$, one has $W^s(x)\cap\Gamma\subset A$.
An ergodic measure $\mu$ is {\rm $W^s$-metrically transitive} if $\mu(A)=0$ or $1$ for any $s$-saturated set $A$.
\end{Definition}
 
\begin{Theorem}\label{Thm:metric-transitive-K}
Assume that $\mu$ is a hyperbolic ergodic regular measure of the flow $\{T^t\}_{t\in\mathbb R}$. If $\mu$ is $W^s$-metrically transitive, then $\mu$ has the K-property.
\end{Theorem}
\begin{proof}
For a partition $\xi$ whose atoms are measurable subsets of $M$, we denote by $\mathcal{B}_\xi$ the $\sigma$-algebra generated by $\xi$, i.e., the minimal $\sigma$-algebra containing all the atoms of $\xi$.
By \cite[Proposition 3.1]{LeS82}
\footnote{The measurable partition in \cite{LeS82} is constructed for $W^u$. A corresponding partition for $W^s$ is obtained by reversing the map.}, there exists a measurable partition $\eta^s$ subordinate to $W^s$ and satisfying the following properties:
\begin{itemize}
\item $\eta^s\le T\eta^s$;
\item $\bigcap_{n\in\mathbb N}\mathcal B_{T^{-n}\eta^s}=\mathcal B_{W^s}$, where $\mathcal B_{W^s}$ is the minimal $\sigma$-algebra containing all stable leaves;
\item the partitions $\{T^n\eta^s\}_{n\in \NN}$ generate $\cB_M$ modulo $\mu$.
\end{itemize}
Let $\mathcal A$ be the $\sigma$-algebra generated by $\eta^s$. The first item implies $T^{-1}\mathcal A\subset\mathcal A$. 
The second item and the assumption $W^s$-metrical transitivity imply ${\rm Tail}(\mathcal A)=\{\emptyset,M\}$ modulo $\mu$. 
The third item implies that $\bigvee_{n=0}^\infty T^n \mathcal A=\mathcal B$ modulo $\mu$. 
We thus conclude by Definition~\ref{Def:Kp}.
\end{proof}
 
In the following section, we will prove that if $\mu$ is a weakly mixing SRB measure, then $\mu$ is $W^s$-metrically transitive.
 
 \subsection{Exceptional set}
Assume that $\mu$ is an ergodic $\chi$-hyperbolic measure for some $\chi>0$, and let $0<\varepsilon_0 \ll \chi$. 
Recall the definition of the Pesin block ${\rm PES}_\ell^{\chi,\varepsilon_0}$ from Section~\ref{Sec:Pesin}, as well as the definition of the local unstable manifolds given in Theorem~\ref{Thm:Pesin-manifold}. 
For simplicity, we denote ${\rm PES}_\ell := {\rm PES}_\ell^{\chi,\varepsilon_0}$.
 
\begin{Proposition}\label{Pro:exceptional-set}
Let $A$ be a measurable set with $\mu(A)>0$. 
For any $\varepsilon>0$ and any $\ell\in\mathbb{N}$, there exists $t_0>0$ such that for all $t\ge t_0$, there is a set $M_\varepsilon^t = M_\varepsilon^t(A,\ell)$ with $\mu(M_\varepsilon^t)<\varepsilon$ such that for any $w,w'\in M\setminus M_\varepsilon^t$ lying on the same local unstable manifold of some point in ${\rm PES}_\ell$, either both $w,w'\in T^t(A)$ or both $w,w'\notin T^t(A)$.
\end{Proposition}
\begin{proof}
For a measurable set $B$, let $\chi_B$ denote its indicator function. 
It is clear that $\chi_{T^{t}B}(x)=\chi_B(T^{-t}(x))$ for all $x\in M$ and all $t\in \RR$.
We will construct a set $M_\varepsilon^t$ with $\mu(M_\varepsilon^t)<\varepsilon$ and show that
\begin{equation}\label{e.indicator}
\sup \{|\chi_{T^{t}A}(w)-\chi_{T^{t}A}(w')|:~w,w'\in W^u_{\rm loc}(z)\setminus M_\varepsilon^t,~z\in {\rm PES}_{\ell} \}=0.
\end{equation}
Note that equation~\eqref{e.indicator} is equivalent to the conclusion of the proposition.

For a measurable set $A$ with $\mu(A)>0$ and any $\varepsilon>0$, we choose a continuous function $f$ and a measurable function $h\in L^2(\mu)$ such that $\chi_A = f + h$ with $\int h^2{\rm d}\mu< \varepsilon^3$. 
By the continuity of $f$, we can choose $t_0>0$ sufficiently large such that 
$$\sup\big\{|f(x)-f(y)|:x,y\in M,~d(x,y)\leq \ell \cdot {\rm e}^{-t_0\chi}\big\}<\varepsilon.$$
For each $t>t_0$, we define the set $M_{\varepsilon}^{t}:=\{x:|h(T^{-t}(x))|>\varepsilon \}$. 
Note that $T^{-t}(M_\varepsilon^t)=M_\varepsilon^0$, where $M_{\varepsilon}^{0}:=\{x:|h(x)|>\varepsilon \}$.
Since $\mu$ is invariant, one has $\mu(M_{\varepsilon}^{t})=\mu(M_{\varepsilon}^{0})$. 
By the choice of $h$, we have that 
$$\varepsilon^3>\int h^2~{\rm d}\mu=\int_{M_\varepsilon^0} h^2~{\rm d}\mu+\int_{M\setminus M_\varepsilon^0} h^2~{\rm d}\mu\ge \varepsilon^2\mu(M_\varepsilon^0).$$ 
Therefore, we have $\mu(M_{\varepsilon}^{t})=\mu(M_{\varepsilon}^{0})<\varepsilon$.

Now, if $w,w'\in W^u_{\rm loc}(z)\setminus M_\varepsilon^t$ for some $z\in {\rm PES}_{\ell}$,  by Theorem~\ref{Thm:Pesin-manifold} one has that $$d(T^{-t}(w),T^{-t}(w'))\leq 2\ell \cdot {\rm e}^{-t\chi}d(w,w'),~\forall t>0.$$ 
Thus, for each $t>t_0$ we have that 
\begin{align*}
|\chi_{T^{t}A}(w)-\chi_{T^{t}A}(w')|&=|\chi_{A}(T^{-t}w)-\chi_{A}(T^{-t}w')|\\
&\le |f(T^{-t}w)-f(T^{-t}w')|+  |h(T^{-t}w)-h(T^{-t}w')|\\
&\leq \varepsilon+2\varepsilon<1.
\end{align*}
Since $|\chi_{T^{t}A}(w)-\chi_{T^{t}A}(w')|=0$ or $1$, it follows that equation~\eqref{e.indicator} holds.
\end{proof}
 
\subsection{$s$-saturated set}
\begin{Proposition}\label{Pro:s-saturated}
Assume that $\mu$ is an ergodic $\chi$-hyperbolic measure for some $\chi>0$.
For any $s$-saturated set $A$ with $\mu(A)\in (0,1)$, there exist an $s$-saturated set $B$ and a time $\tau>0$ such that $\mu(T^{[-\tau,\tau]}(B))\in (0,1)$.
\end{Proposition}
\begin{Remark}
Since $B$ is $s$-saturated, it follows that $T^{[-\tau,\tau]}(B)$ is also $s$-saturated.
\end{Remark}
\begin{proof}[Proof of Proposition~\ref{Pro:s-saturated}]
Let $D=M\setminus A$, then $D$ is also $s$-saturated and $\mu(D)\in (0,1)$.  Define 
$$m_{\tau}^{A}(x):=\frac{1}{2\tau} {\rm Leb} \big(\{t \in [-\tau,\tau]:T^{t}(x)\in A \} \big),~~~
  m_{\tau}^{D}(x):=\frac{1}{2\tau} {\rm Leb} \big(\{t \in [-\tau,\tau]:T^{t}(x)\in D \} \big).$$
Note that the function $(x,t)\mapsto \chi_{ \{(x,t):~|t|\leq \tau,~T^t(x)\in A \}}$ is measurable\footnote{Consider the continuous map $F:~(x,t)\mapsto (T^t(x),t)$. It is clear that $\{(x,t):~|t|\leq \tau,~T^t(x)\in A \}=F^{-1}(A\times[-\tau,\tau])$. 
Thus, this set is measurable.}. 
By Fubini's theorem, 
$$m_{\tau}^{A}(x)=\frac{1}{2\tau}\int_{-\tau}^{\tau} \chi_{\{ t:~|t|\leq \tau,~T^t(x)\in A \}} {\rm d}t$$ 
is a measurable function on $M$. 
Clearly, $0\leq m_{\tau}^{A}(x)\le 1$ for all $x$.
\begin{Claim}
When $\tau \to 0$, we have $m_{\tau}^{A}(x)\to 1$ for $\mu$-almost every $x\in A$, and $m_{\tau}^{D}(x)\to 1$ for $\mu$-almost every $x\in D$.
\end{Claim}
\begin{proof}[Proof of the Claim]
By the definition of $m_{\tau}^{A}$, we have
\begin{align*}
   \int_{A} m_{\tau}^{A}(x) {\rm d}\mu(x)&=\frac{1}{2\tau} \int_{M} \int_{[-\tau,\tau]} \chi_{A}(T^{t}x) \chi_A(x)~{\rm d}t~{\rm d}\mu(x)\\ &=\frac{1}{2\tau} \int_{[-\tau,\tau]} \int_{M}  \chi_{A}(T^{t}x) \chi_A(x)~{\rm d}\mu(x)~{\rm d}t.
\end{align*}
By the dominated convergence theorem, we have 
$$\lim_{t\to 0} \int_{M}  \chi_{A}(T^{t}x) \chi_A(x)~{\rm d}\mu(x)=\int_{M}  \chi_A(x)^{2}~{\rm d}\mu(x)=\mu(A).$$
Therefore, we have 
$$\int_{A} \lim_{\tau\to 0} m_{\tau}^{A}(x) {\rm d}\mu(x)=\lim_{\tau\to 0} \int_{A} m_{\tau}^{A}(x) {\rm d}\mu(x)=\mu(A),$$
where the first equality follows from the dominated convergence theorem. 
The above limit implies the conclusion of the claim, since $0 \le m_{\tau}^{A}(x) \le 1$ for all $x\in M$ and all $\tau>0$.
Since $\mu(D)\in (0,1)$, the same conclusion holds for $D$.
\end{proof}

Since $A$ is $s$-saturated and $T^{t}(x)$ and $T^{t}(y)$ lie on the same stable leaf whenever $x$ and $y$ do, it follows that, when restricted to a full $\mu$-measure set, $m_{\tau}^{A}(x)$ is constant along stable leaves. 
A similar statement holds for $m_{\tau}^{D}(x)$.
Choose $\tau>0$ small enough such that
\begin{itemize}
\item $B:=\{x\in A: m_{\tau}^{A}(x)>9/10\}$ has positive $\mu$-measure;
\item $E:=\{x\in D: m_{\tau}^{D}(x)>9/10\}$ has positive $\mu$-measure.
\end{itemize}
The first item implies that $\mu(T^{[-\tau,\tau]}(B))>0$.

Note that $B$ and $E$ are measurable, and they are $s$-saturated since $m_{\tau}^{A}(x)$ and $m_{\tau}^{D}(x)$ are constant on stable leaves. 
We now prove that $\mu(T^{[-\tau,\tau]}(B))<1$. 
Since $\mu(E)>0$, it suffices to show that $T^{[-\tau,\tau]}(B)\cap E=\emptyset$.

If $T^{[-\tau,\tau]}(B)\cap E\neq \emptyset$, then there exists a point $x \in T^{[-\tau,\tau]}(B)\cap E$. 
We assume that there is $s \in[-\tau,\tau]$ such that $T^s(x)\in B$.
Since $x\in E$, one has 
\begin{equation}\label{e.inequa-E}
{\rm Leb} \big(\{t \in [-\tau,\tau]:T^{t}(x)\notin A \}\big) \geq \frac{9\tau}{5}.
\end{equation}
Since $T^s(x)\in B$, one has that ${\rm Leb} (\{t \in [-\tau,\tau]:T^{t+s}(x)\notin A \}) \leq \frac{\tau}{5}$. 
Thus, we have 
\begin{equation}\label{e.inequa-B}
{\rm Leb} \big(\{t \in [-\tau,\tau]:T^{t}(x)\notin A \}\big) \leq \tau+\frac{\tau}{5}\leq \frac{6\tau}{5}.
\end{equation}
Equations~\eqref{e.inequa-E} and \eqref{e.inequa-B} give a contradiction together.
Therefore, we have $T^{[-\tau,\tau]}(B)\cap E= \emptyset$.
This completes the proof of the proposition.
\end{proof}
 
\subsection{Proof of Theorem~\ref{Thm:weakly-to-K}}
Now assume that $\mu$ is a weakly mixing hyperbolic SRB measure for a $C^{1+\beta}$ flow $\{T^t\}_{t\in\mathbb{R}}$. 
To prove that $\mu$ satisfies the K-property, by Theorem~\ref{Thm:metric-transitive-K} it suffices to show that $\mu$ is $W^s$-metrically transitive.
Fix $\chi>0$ and $0<\varepsilon_0 \ll \chi$ such that $\mu$ is $\chi$-hyperbolic. 
Recall the definition of the Pesin block $\PES_{\ell}^{\chi,\varepsilon_0}$ from Section~\ref{Sec:Pesin}, and write ${\rm PES}_{\ell}:= {\rm PES}_{\ell}^{\chi,\varepsilon_0}$ for simplicity.
\begin{proof}[Proof of Theorem \ref{Thm:weakly-to-K}]
	We argue by contradiction and assume that there exists an $s$-saturated set $A$ such that $\mu(A)\in(0,1)$. 
	By Proposition~\ref{Pro:s-saturated}, there exist an $s$-saturated set $B$ with $\mu(B)>0$ and $\tau>0$ such that $\mu(C)\in(0,1)$, where $C = T^{[-\tau,\tau]}(B)$. 
	Now, we choose
	\begin{itemize}
		\item $\ell \in \mathbb{N}$ such that $\mu({\rm PES}_\ell) > \mu(C)$.
	\end{itemize}
	By Lemma~\ref{Lem:transverse-intersection}, for the given $\ell\in \NN$ there exists $L\in \NN$ satisfying the conclusion of Lemma~\ref{Lem:transverse-intersection}.
	By Theorem~\ref{Thm:Pesin-manifold}, for the larger Pesin block ${\rm PES}_L$, there exists $\delta(L)>0$ such that for every $x\in {\rm PES}_L$, the local stable and unstable manifolds $W^{s}_{\delta(L)}(x)$ and $W^{u}_{\delta(L)}(x)$ of size $\delta(L)$ are well defined.
    Take $0<\alpha\ll \min\{\tau,\delta(L)\}$ sufficiently small such that
    \begin{itemize}
    	\item for every $x,y\in \PES_L$ with $d(x,y)<\alpha$, the map $h^{cs}_{x,y}:\PES_{L}\cap W^{u}_{\delta(L)}(x) \to W^u_{\loc}(y)$ is well-defined, $\Jac(h^{cs}_{x,y})\in (1/2,2)$ and $h^{cs}_{x,y}(z)\in T^{[-\tau,\tau]}W^s_{\loc}(z)$ for any $z\in \PES_L\cap W^{u}_{\delta(L)}(x)$;
    	\item there is a finite collection of pairwise disjoint $(\alpha,u)$-rectangles $\{\Pi^u_1,\cdots,\Pi^u_{N_*}\}$ at points $\{w_1,\cdots,w_{N_*}\}$ such that $ \bigcup_{i=1}^{N_*} \Pi^u_{i}\subset {\rm PES}_{L}$, $\mu( \bigcup_{i=1}^{N_*} \Pi^u_{i} \setminus {\rm PES}_{\ell})=0$ and $\mu(\Pi^u_{i}\cap \PES_\ell)>0$ for any $ 1\le i\le N_*$.
    \end{itemize}   
	By Corollary~\ref{Pro:density-zero}, there exist $\alpha_0>0$ and a sequence of $\{t_n\}\subset\mathbb R$ such that $t_n\to +\infty$ as $n\to \infty$ and 
	$$\mu(T^{t_n}(B)\cap \Pi^u_i)>\alpha_0~~ \text{for any}~n\geq 1~\text{and any}~1\le i\le N_*.$$ 
	\begin{Proposition}\label{Prop:main-K}
		For each $1\le i\le N_*$, one has
		$$\lim_{n\to\infty} \mu(\Pi_i^u \setminus T^{t_n}(C))=0.$$
	\end{Proposition}
	We first assume that this proposition holds and give the proof of Theorem \ref{Thm:weakly-to-K}, we will then prove the proposition.
	Note that $\mu(C)=\mu(T^{t_n}(C))$. 
	Then, we have 
	$$\mu({\rm PES}_\ell\setminus T^{t_n}(C))\le \mu\Big(\bigcup_{i=1}^{N_*} \Pi^u_i\cap {\rm PES}_\ell\setminus T^{t_n}(C)\Big)\le\sum_{i=1}^N\mu(\Pi^u_i\setminus T^{t_n}(C))\to 0.$$
	This implies that $\mu(C)=\lim\limits_{n\to \infty}\mu(T^{t_n}(C))\ge \mu({\rm PES}_\ell)$ which contradicts $\mu({\rm PES}_\ell)>\mu(C)$.
\end{proof}

\begin{proof}[Proof of Proposition \ref{Prop:main-K}]
For each $1\leq i \leq N_*$, we define a partition $\xi_i^u$ of the positive $\mu$-measure subset
$$S_i:=\bigcup_{x\in \Pi_i^u} W_{\delta(L)}^u(x)\cap B(w_i,\alpha\varepsilon(L)^{-1}).$$ 
For each $y\in S_i$, choose $x_y\in \Pi_i^u$ and define $\xi_i^u(y):=W_{\delta(L)}^u(x_y)\cap B(w_i,\alpha\varepsilon(L)^{-1})$.
It is clear that when $\alpha$ is small enough, $\xi_i^u(y)$ does not depend on the choice of $x_y$.
Since the local unstable manifolds depend continuously on $\PES_\ell$, it follows that the partition $\xi_i^u$ is a measurable partition of $S_i$ and $\Leb_{W^u(x)}(\xi_i^u(x)) / \Leb_{W^u(y)}(\xi_i^u(y))\leq 2$ for every $x,y\in \Pi_i^u$.
Let $\mu|_{S_i}(\cdot)=\mu(S_i\cap \cdot)$, and denote by $\mu|_{S_i}=\int_{S_i} \mu_{\xi_i^u(x)}~{\rm d} \mu(x)$ the canonical disintegration of of $\mu|_{S_i}$ with respect to the measurable partition $\xi_i^u$.

Since $\mu$ is an SRB measure, $\mu_{\xi_i^u(x)}$ is absolutely continuous with respect to ${\rm Leb}_{W^u(x)}$ for every $x\in \Pi_i^u$.
Moreover, by Theorem~\ref{Thm:SRB}, there exists a constant $K>0$ such that for every $x\in \Pi_i^u$,
$$K^{-1}\le \frac{{\rm d}\mu_{\xi_i^u(x)}(y)}{{\rm d}{\rm Leb}_{\xi_i^u(x)}(y)}\le K,~~~\forall y\in \xi_i^u(x),$$
where ${\rm Leb}_{\xi_i^u(x)}(\cdot)={\rm Leb}_{W^u(x)}(\xi_i^u(x) \cap \cdot) / {\rm Leb}_{W^u(x)}(\xi_i^u(x))$.
 
\begin{Lemma}\label{Lem:Leb-pes}
For every $\varepsilon>0$, for all sufficiently large $t_n$, and for each $1\le i\le N_{*}$, there exists a point $x\in T^{t_n}(B)\cap \Pi_i^u$ such that
\begin{equation}\label{e.one-point-good-rectangle}
{\rm Leb}_{\xi_i^u(x)} \big( \Pi_i^u \setminus T^{t_n}(B) \big)< \varepsilon.
\end{equation}
\end{Lemma}
\begin{proof}[Proof of Lemma \ref{Lem:Leb-pes}]
For $\varepsilon>0$, set 
$$\varepsilon'=K^{-2}\varepsilon \alpha_0 \cdot \min_{i=1,\cdots,N_*} \mu(\Pi_{i}^u).$$
For the measurable set $B$ and the constant $L$, by Proposition~\ref{Pro:exceptional-set}, for every sufficiently large $t_n$ there exists an exceptional set $M_{\varepsilon'}^{t_n}=M_{\varepsilon'}^{t_n}(B,L)$ such that $\mu(M_{\varepsilon'}^{t_n})<\varepsilon'$.

By Proposition~\ref{Pro:exceptional-set}, for any $x\in {\rm PES}_L\cap T^{t_n}(B)$, one has $W^u_{\rm loc}(x)\cap \Pi_{i}^u\setminus M_{\varepsilon'}^{t_n} \subset T^{t_n}(B)$, since $x\in T^{t_n}(B)$ and $\Pi_{i}^u\subset \PES_{L}$. 
Therefore, we have 
$$W^u_{\rm loc}(x)\cap \Pi_{i}^u\setminus T^{t_n}(B)\subset W^u_{\rm loc}(x) \cap \Pi_{i}^u \cap M_{\varepsilon'}^{t_n}.$$
It suffices to find a point $x\in T^{t_n}(B)\cap \Pi_i^u$ such that 
$${\rm Leb}_{\xi_i^u(x)}\big( \Pi_i^u \setminus T^{t_n}(B) \big)< \varepsilon.$$
If no such point exists, then for every $x\in T^{t_n}(B)\cap \Pi_i^u$, one has
$${\rm Leb}_{\xi_i^u(x)}\big(\Pi_i^u \cap M_{\varepsilon'}^{t_n} \big) \ge {\rm Leb}_{\xi_i^u(x)} \big( \Pi_i^u \setminus T^{t_n}(B) \big) \ge \varepsilon.$$ 
By the SRB property of $\mu$, one has that 
\begin{align*}
\mu(M_{\varepsilon'}^{t_n})&\ge\mu(S_i  \cap M_{\varepsilon'}^{t_n})\\
&\geq \int_{\Pi_i^u\cap T^{t_n}(B)  }\mu_{\xi_i^u(x)}(\Pi_i^u \cap M_{\varepsilon'}^{t_n})~{\rm d}\mu(x)\\
&\ge K^{-1}\int_{\Pi_i^u\cap T^{t_n}(B) }{\rm Leb}_{\xi_i^u(x)}(\Pi_i^u \cap M_{\varepsilon'}^{t_n})~{\rm d}\mu(x)\\
&\ge K^{-2}\varepsilon \cdot \mu(\Pi_i^u\cap T^{t_n}(B) )\ge \varepsilon'.
\end{align*}
This contradicts the fact that $\mu(M_{\varepsilon'}^{t_n})<\varepsilon'$.
\end{proof}
  
Fix $1\le i\le N_{*}$. 
For each $\varepsilon>0$, by Lemma \ref{Lem:Leb-pes} for all sufficiently large $t_n$ there exists $x\in T^{t_n}(B)\cap \Pi_i^u$ such that 
\begin{equation}
{\rm Leb}_{W^u(x)}\big(\Pi_i^u \setminus T^{t_n}(B)\big)< \varepsilon \cdot {\rm Leb}_{W^u(x)}(\Pi_i^u).
\end{equation}
By the fact that $B$ is $s$-saturated, $C=T^{[-\tau,\tau]}(B)$ and $\alpha \ll \tau$, for every $z\in \Pi_i^u\subset {\rm PES}_L$, if $y\in \Pi_i^u \cap T^{t_n}(B)$, then we have 
$h^{cs}_{x,z}(y)\in T^{t_n}(C)$.
Hence,  $h^{cs}_{x,z}(\Pi_i^u\cap T^{t_n}(B))\subset T^{t_n}(C)$.
By $\Pi_{i}^u$ is an $u$-rectangle, we have $h^{cs}_{x,z}(W^u_{\loc}(x)\cap \Pi_i^u)=W^u_{\loc}(z)\cap \Pi_i^u$. 
Therefore, 
$$W^u_{\loc}(z)\cap \Pi_i^u \setminus T^{t_n}(C)\subset h^{cs}_{x,z}(\Pi_i^u \setminus T^{t_n}(B)).$$
Thus, by Theorem~\ref{Thm:holonomy}, we have that
\begin{align*}
	{\rm Leb}_{W^u(z)}(\Pi_{i}^u \setminus T^{t_n}(C))\leq {\rm Leb}_{W^u(z)}\big(h^{cs}_{x,z}(\Pi_i^u \setminus T^{t_n}(B))\big) \leq 2{\rm Leb}_{W^u(x)}(\Pi_i^u \setminus T^{t_n}(B)).
\end{align*}
Therefore, for every $z\in \Pi_i^u$ we have 
$${\rm Leb}_{\xi_i^u(z)}(\Pi_i^u \setminus T^{t_n}(C))=\frac{{\rm Leb}_{W^u(z)}(\Pi_{i}^u \setminus T^{t_n}(C))}{{\rm Leb}_{W^u(z)}(\xi_i^u(z))}\leq \frac{4{\rm Leb}_{W^u(x)}(\Pi_{i}^u\setminus T^{t_n}(B))}{{\rm Leb}_{W^u(x)}(\xi_i^u(x))}\leq 
4\varepsilon.$$
Thus, for all sufficiently large $t_n$ we have that 
\begin{align*}
\mu(\Pi_i^u\setminus T^{t_n}(C))
    &\leq \int \mu_{\xi_i^u(z)}(\Pi_i^u\setminus T^{t_n}(C)){\rm d}\mu(z)\\
	&\leq K\int {\rm Leb}_{\xi_i^u(z)}(\Pi_i^u \setminus T^{t_n}(C))~{\rm d}\mu(z)\\
	&\leq 4\varepsilon K.
\end{align*}
This completes the proof of Proposition \ref{Prop:main-K}.
\end{proof}

\section{Proof of Theorem~\ref{Thm:K-to-B}: the B-property}
The key point in proving Bernoulli property is to establish the very weak Bernoulli property. 
We first recall its definition.

\subsection{The distance $\bar d$}
Assume that $(X,\mathcal B,\mu)$ is a probability space. 
Given a measurable set $P$ with $\mu(P)>0$, we define the conditional measure $\mu_P$ by 
$$\forall B\in\mathcal B, ~~~\mu_P(B):=\frac{\mu(P\cap B)}{\mu(P)}.$$
and define the $\sigma$-algebra $\mathcal{B}_P$ on $P$ by $\mathcal{B}_P:=\{B\cap P:~B\in\mathcal B\}$.
Thus, we obtain a new probability space $(P,\mathcal B_P,\mu_P)$. 
Given a finite measurable partition $\mathcal P=\{P_1,P_2,\cdots,P_n\}$ of $X$, we define a partition of the probability space $(P,\mathcal B_P,\mu_P)$ by:
$$\mathcal P|P:=\{P_1\cap P,~P_2\cap P,~\cdots,~P_n\cap P\}.$$
 
For an ordered partition $\mathcal P=\{P_1,\cdots,P_N\}$ of $(X,\mathcal B,\mu)$ and for $x\in X$, define $\mathcal P_x=i$ if $x\in P_i$.
For any two ordered partitions $\mathcal P=\{P_1,\cdots,P_N\}$ and $\mathcal Q=\{Q_1,\cdots,Q_N\}$, we define the partition distance by
$$d(\mathcal P,\mathcal Q):=\sum_{i=1}^N\mu(P_i \triangle Q_i) = 2 \int \mathbb{1}_{[\mathcal P(x) \neq \mathcal Q(x)]} {\rm d}\mu(x).$$

Assume that $\{\mathcal P^i\}_{i=1}^n$ and $\{\mathcal Q^i\}_{i=1}^n$ are finite sequences of ordered partitions of $(X,\mathcal B,\mu)$ and $(Y,\mathcal C,\nu)$, respectively, and that each partition consists of $N$ elements for some $N\in\mathbb N$. 
We write
$$\mathcal P^i=\{P_1^i,\cdots,P_N^i\},~~\mathcal Q^i=\{Q_1^i,\cdots,Q_N^i\}.$$
We say that $\{\mathcal P^i\}_{i=1}^n$ and $\{\mathcal Q^i\}_{i=1}^n$ have the same distribution, and write $\{\mathcal P^i\}_{i=1}^n \sim \{\mathcal Q^i\}_{i=1}^n$, if
$$\mu(P_{i_1}\cap \cdots P_{i_n})=\nu(Q_{i_1}\cap \cdots Q_{i_n}),\quad \forall (i_1,\cdots,i_n)\in \{1,\cdots,N\}^n. $$
Two sequences of ordered partitions $\{\mathcal P^i\}_{i=1}^n$ and $\{\mathcal Q^i\}_{i=1}^n$ have the same distribution if and only if there exists a measure preserving map $\theta: (X,\mathcal B,\mu)\to (Y,\mathcal C,\nu)$ such that
$$\theta(P_{i_1}^1\cap\cdots\cap P_{i_n}^n)=Q_{i_1}^1\cap\cdots\cap Q_{i_n}^n~~~\textrm{mod}~\nu,~~~\forall (i_1,\cdots,i_n)\in\{1,\cdots,N\}^n.$$

\begin{Definition}\label{Def:d-bar-distance}
Given $\{\mathcal P^i\}_{i=1}^n$ and $\{\mathcal Q^i\}_{i=1}^n$, we define their $\bar{d}$-distance by
\begin{equation*}
	\bar{d}(\{\mathcal P^i\}_{i=1}^n,\{\mathcal Q^i\}_{i=1}^n):= \inf \left\{ \frac{1}{n}\sum_{i=1}^n d({\mathcal P}^i,\mathcal Q^i):~~
	\begin{aligned}
		&\{{\mathcal P}^i\}_{i=1}^n~\textrm{are ordered partitions of~} \\
		&(Y,\mathcal C,\nu)~\textrm{such that}~\{{\mathcal P}^i\}_{i=1}^n\sim \{\mathcal Q^i\}_{i=1}^n
	\end{aligned}	
	\right \}.
\end{equation*}
\end{Definition}

Let $\theta:(X,\mathcal B,\mu)\to (Y,\mathcal C,\nu)$ be an invertible measurable map. 
Given $\varepsilon>0$, we say that $\theta$ is \textit{$\varepsilon$-measure preserving} if there exists a set $E\in\mathcal B$ with $\mu(E)<\varepsilon$ such that $\big|\frac{\nu(\theta(A)))}{\mu(A)}-1 \big|\le\varepsilon$ for any measurable set $A\subset X\setminus E$.
\begin{Remark}
In practical application, since $\varepsilon$ is an arbitrary small number, one only needs to define $\theta$ on a set with measure larger than or equal to $1-\varepsilon$. 
If $\theta:(X,\mathcal B,\mu)\to (Y,\mathcal C,\nu)$ is $\varepsilon$-measure preserving, then it follows from the definition that the inverse map $\theta^{-1}:(Y,\mathcal C,\nu)\to (X,\mathcal B,\mu)$ is $\varepsilon'$-measure preserving for some $\varepsilon'$ of the same order as $\varepsilon$.
\end{Remark}

We need the following lemma of Ornstein-Weiss \cite{OrW73} to estimate the $d$-bar distance.

\begin{Lemma}\label{Lem:Ornstein}
Assume that $\{\mathcal P^i\}_{i=1}^n$ and $\{\mathcal Q^i\}_{i=1}^n$ are two ordered partitions and $\theta:(X,\mathcal B,\mu)\to (Y,\mathcal C,\nu)$ is $\varepsilon$-measure preserving for some $\varepsilon>0$.
If 
$$\mu\left(\Big\{x:~\frac{1}{n}\sum_{i=1}^n \mathbb{1}_{[\mathcal P^i_x\neq\mathcal Q^i_{\theta(x)}]}(x)\le\varepsilon\Big\}\right)\ge 1-\varepsilon,$$
then $\bar d(\{\mathcal P^i\}_{i=1}^n,\{\mathcal Q^i\}_{i=1}^n)\le 16\varepsilon$, where $\mathcal P^i_x=j$ if $x\in P^i_j\in \cP^i$.
\end{Lemma}

\subsection{Very weak Bernoulli partitions}
Let $(X,\mathcal B,\mu,T)$ be a measure preserving automorphism and let $\cP=\{P_1,\cdots,P_N\}$ be a finite measurable partition.
For $\varepsilon>0$, We say that $\varepsilon$-almost every $P\in \mathcal{P}$ has the property $\ast$, when $\ast$ holds for a set of atoms with total measure  $\geq 1-\varepsilon$.

If $\cP=\{P_1,\cdots,P_N\}$ is an ordered partition, then for each $i \in \ZZ$, $T^{-i}\cP:=\{T^{-i}P_1,\cdots, T^{-i}P_N\}$ is also an ordered partition.
Moreover, for any $E\in \mathcal B$, the restriction $\mathcal{P}|E:=\{P_1\cap E,\cdots,P_N\cap E\} $ is also an ordered partition.
\begin{Definition}
Let $(X,\mathcal B,\mu,T)$ be a measure preserving automorphism and let $\cP$ be a finite measurable ordered partition of $X$.
We say that $\mathcal P$ is very weak Bernoulli (VWB) if for every $\varepsilon>0$, there is $N_0=N_0(\varepsilon)$ such that for every $n\ge 0$ and every $N'\ge N\ge N_0$ the following holds: for $\varepsilon$-almost every $P\in\bigvee_{k=N}^{N'}T^k\mathcal P$ we have
\begin{equation}\label{e.wvb}
\bar d\big(\{T^{-i}\mathcal P\}_{i=1}^n,~\{T^{-i}\mathcal P|P\}_{i=1}^n\big)<\varepsilon,
\end{equation}
where $\{T^{-i}\mathcal P\}_{i=1}^n$ is a sequence of ordered partitions in $(X,\mathcal B,\mu)$ and $\{T^{-i}\mathcal P|P\}_1^n$ is a sequence of ordered partitions in $(P,~\mathcal B_P,~\mu_P)$ 
\end{Definition}

\begin{Theorem}[Ornstein's theorem, \cite{Orn70,Orn70b,OrW73}]\label{Thm:Ornstein}
Assume that $(X,\mathcal B,\mu,\{T^t\})$ is a probability-preserving flow.
If for some $t_0 \neq 0$, the system $(X,\mathcal B,\mu,T^{t_0})$ admits an increasing sequence of very weak Bernoulli partitions that generate $\mathcal B$, then $\mu$ is Bernoulli.
\end{Theorem}

Given a partition $\mathcal P=\{P_1,\cdots,P_n\}$, define $\partial \mathcal P=\bigcup_{i=1}^n\partial P_i$ and ${\rm Diam}\mathcal P=\max_{i}{\rm Diam}(P_i)$. 
The main theorem in this section is:
\begin{Theorem}\label{Thm:partition-VWB}
	Assume that $\mu$ is a hyperbolic SRB measure for a flow $\{T^t\}_{t\in\mathbb R}$, and that $\mu$ has the K-property.
	Let $\lambda>0$ be sufficiently small, and let $\mathcal P$ be a finite partition with piecewise smooth boundary such that
	\begin{equation}\label{e.super-regular}
		\sum_{n=1}^{\infty} \mu\big( B(\partial \mathcal P, e^{-n\lambda}) \big)<\infty.
	\end{equation}
	Then $\mathcal P$ is a very weak Bernoulli partition for $(M,\mathcal{B},\mu,T^{-1})$.
\end{Theorem}
Recall that in Theorem \ref{Thm:weakly-to-K} we have shown that a weakly mixing hyperbolic SRB measure has the K-property. 
Combining the above theorem with Ornstein's theorem (Theorem \ref{Thm:Ornstein}), we can now give a quick proof of Theorem \ref{Thm:K-to-B}.

\begin{proof}[Proof of Theorem~\ref{Thm:K-to-B}]
From \cite[Proposition 3.2]{LeS82}, there exists an increasing sequence of finite partitions $\{\mathcal P_m\}_{m>0}$ with piecewise smooth boundaries such that
\begin{itemize}
\item ${\rm Diam}(\mathcal{P}_m)\to 0$ as $m\to\infty$;
\item \eqref{e.super-regular} holds for each $\mathcal P_m$.
\end{itemize}
By Theorem~\ref{Thm:partition-VWB}, each $\mathcal P_m$ is a very weak Bernoulli partition for $(M,\mathcal{B},T^{-1},\mu)$.
The condition ${\rm Diam}(\mathcal P_m)\to 0$ implies that $\{\mathcal P_m\}$ generates $\mathcal B$.  
Therefore, by Theorem~\ref{Thm:Ornstein}, the flow $(M,\mathcal{B},\mu,\{T^t\}_{t\in\mathbb R})$ is Bernoullian.
\end{proof}

\subsection{Proof of Theorem \ref{Thm:partition-VWB}}
To verify that the partition $\mathcal P$ in Theorem~\ref{Thm:partition-VWB} is very weak Bernoulli, 
By Lemma~\ref{Lem:Ornstein}, it suffices to show that for every $\varepsilon>0$, there exists $N_0>0$ such that for any $n>0$, any $N>N'>N_0$, and $\varepsilon$-almost every $P\in \bigvee_{n=N'}^{N} T^{-n} \mathcal P$, there exists a $13\varepsilon$-measure preserving map $\theta:(P,~\mathcal{B}_{P},\mu_P)  \rightarrow (M,\mathcal{B},\mu)$ such that
\begin{equation}\label{e.precise-formula}
\mu\left(\Big\{x:\frac{1}{n}\sum_{j=0}^{n-1}\mathbb{1}_{[(T^j\mathcal P)_{x}\neq (T^j\mathcal P|P)_{(\theta^{-1}x)}]}(x)\leq \varepsilon \Big\}\right)\geq 1-\varepsilon.
\end{equation}
In this case, the $13\varepsilon$-measure preserving property can be equivalently formulated as follows: a map $\theta:(P,~\mathcal{B}_{P},\mu_P)  \rightarrow (M,\mathcal{B},\mu) $ is $\varepsilon$-measure preserving if and only if there is a set $E\in \mathcal{B}_P$ with $\mu_P(E)=\frac{\mu(E)}{\mu(P)}>1-\varepsilon$ such that for any $B\subset E$, one has
\begin{equation}\label{e.13-varpesilon}
	\left|\frac{\mu(\theta(B))}{ \mu_P(B)}-1\right|=\left|\frac{\mu(\theta(B))\cdot\mu(P)}{ \mu(B)}-1\right|\leq 13\varepsilon.
\end{equation}
In the remainder of the paper, we will establish the existence of such a $13\varepsilon$-measure preserving map, thereby completing the proof of Theorem~\ref{Thm:partition-VWB}

\subsubsection{Leaf-wise intersections}
Assume that $\mu$ is an ergodic $\chi$-hyperbolic measure for some $\chi>0$, and let $0<\varepsilon_0 \ll \chi$. 
Recall the definition of the Pesin block ${\rm PES}_\ell^{\chi,\varepsilon_0}$ from Section~\ref{Sec:Pesin}, and for simplicity write ${\rm PES}_\ell := {\rm PES}_\ell^{\chi,\varepsilon_0}$.

We also recall the definition of an $s$-rectangle from Definition~\ref{Def:rectangle}.
A set E is said to intersect a $s$-rectangle $\Pi^s$ \textit{$s$-leaf-wise}, if for every $x\in E$, one has
$$W^{s}_{\mathrm{loc}}(x)\cap \Pi \subset W^{s}_{\mathrm{loc}}(x)\cap E.$$
In this subsection, we only consider $s$-rectangles. 
So we simply refer to them as rectangles when no confusion arises.

\begin{Proposition}\label{Pro:s-leaf-wise}
Assumes that $\lambda\in(0,\chi/2)$, and let $\mathcal P$ be the partition as in Theorem~\ref{Thm:partition-VWB}. 
Then, for each $\varepsilon>0$ and every $s$-rectangle $\Pi\subset \PES_{L}$, there exists $N_0\in \NN$ such that for any $N'>N\geq N_0$ and for $\varepsilon$-almost every $P\in\bigvee_{n=N}^{N'}T^{-n}\mathcal P$, there exists $E\subset P$ which intersects $\Pi$ $s$-leaf-wise and $\mu_P(E)=\mu(E)/\mu(P)\ge 1-\varepsilon$.
\end{Proposition}
\begin{proof}
By Theorem \ref{Thm:Pesin-manifold}, for every $x\in \PES_{L}$ and every $y,z\in W^s_{\loc}(x)$, we have
$$d(T^n(y),T^n(z))\leq  2L e^{-n\chi/2} d(y,z)\leq 2L e^{-n\chi/2}.$$ 
Given $\varepsilon>0$, together with condition~\eqref{e.super-regular}, we choose $N_0\in \mathbb{N}$ such that
$$\sum_{n=N_0}^{+\infty} \mu(B(\partial \mathcal P,{\rm e}^{-\lambda n}) )\leq \varepsilon^2,~~\text{and}~~2L e^{-N_0\chi/2}\leq e^{-N_0\lambda}.$$
Define
$$F=\bigcup_{n=N_0}^\infty T^n(B(\partial\mathcal P,{\rm e}^{-\lambda n})).$$
Since $\mu$ is an invariant measure, it follows that 
$$\mu(F)\le\sum_{n=N_0}^\infty\mu(T^n(B(\partial\mathcal P,{\rm e}^{-\lambda n})))=\sum_{n=N_0}^\infty\mu(B(\partial\mathcal P,{\rm e}^{-\lambda n}))<\varepsilon^2.$$
Fix $N>N'>N_0$.
Consider all $P\in \bigvee_{n=N}^{N'}T^{-n}\mathcal P$ such that $\mu(P\cap F)>\varepsilon\mu(P)$.
Then
\begin{align*}
	\mu\left(\bigcup \big\{P\in \bigvee_{n=N}^{N'}T^{-n}\mathcal P: \mu(P\cap F)>\varepsilon\mu(P) \big\}\right)
	\le \frac{1}{\varepsilon} \sum_{P}\mu(P\cap F)\le \frac{1}{\varepsilon}\mu(F)<\varepsilon.
\end{align*}
Therefore, $\mu(P\cap F)\leq \varepsilon\mu(P)$ for $\varepsilon$-almost every $P\in \bigvee_{n=N}^{N'}T^{-n}\mathcal P$.

For each $P\in\bigvee_{n=N}^{N'}T^{-n}\mathcal P$, we write $P=P_N\cap\cdots\cap P_{N'}$, where $P_n\in T^{-n}\mathcal P$, $n\in[N',N]$. 
Let $E_n\subset P_n$ be the maximal subset that $E_n\cap\Pi$ $s$-leaf-wise (i.e., if $E'\subset P_n$ also intersects $\Pi$ $s$-leaf-wise, then $E'\subset E_n$). 
By maximality, we obtain the following: 
\begin{itemize}
	\item if $x\in W^s_{\loc}(y)\cap P_n$ for some $y\in \Pi$, then $W^s_{\loc}(y)\cap  P_n \subset E_n$;
	\item if $x\in P_n \setminus  \bigcup_{y\in \Pi} W^s_{\loc}(y)$, then we have $x\in E_n$.
\end{itemize}
Since the boundary of $P_n$ is smooth, it follows that for each $z\in P_n\setminus E_n$, there exist $x\in \partial P_n$ and $y\in\Pi$ such that $x,z\in W^s_{\rm loc}(y)$.
Hence, we have
$$d(T^n(z),T^n(x))\leq  2L{\rm e}^{-\chi n} <{\rm e}^{-\lambda n}$$
which implies that $T^n(z)\in  B(\partial\mathcal P,{\rm e}^{-\lambda n})$. 

Let $E=E_{N'}\cap \cdots \cap E_{N}$.
Clearly, $E$ intersects $\Pi$ $s$-leaf-wise.
Recall that $\mu(P\cap F)\leq \varepsilon\mu(P)$ for $\varepsilon$-almost every $P\in \bigvee_{n=N}^{N'}T^{-n}\mathcal P$.
Since
$$\mu(P\setminus E) \leq \mu(P\cap \bigcup_{n=N'}^{N}T^nB(\partial \mathcal P,{\rm e}^{-\lambda n}) )\leq \mu(P\cap F),$$
it follows that $\mu(P\setminus E)\leq \varepsilon \mu(P)$, or equivalently, $\mu_P(E)=\frac{\mu(E)}{\mu(P)} \ge 1-\varepsilon$.
This completes the proof of the proposition.
\end{proof}

Given $\varepsilon>0$, we choose $\ell\in\mathbb N$ such that $\mu({\rm PES}_{\ell})>1-\varepsilon^2$.
By Lemma~\ref{Lem:transverse-intersection}, for the given $\ell\in \NN$ there exists $L\in \NN$ satisfying the conclusion of Lemma~\ref{Lem:transverse-intersection}.
Let $\alpha>0$ be a sufficiently small number.
By Proposition~\ref{Pro:rectangle-covering}, there are disjoint $(\alpha,s)$-rectangles $\{\Pi_i^s\}_{i=1}^{N_*}$ satisfying 
$$\PES_{\ell}\subset \bigcup_{i=1}^{N_*} \Pi_i^s\subset \PES_{L}.$$
Moreover, we may assume that $\mu(\Pi_i^s)>0$ for all $1\le i\le N_*$.

\subsubsection{Local $\varepsilon$-measure preserving map $\theta_\Pi$}

\begin{Proposition}\label{Pro:local-map}
Let $~\Pi\in \{\Pi_i^s\}_{i=1}^{N_*}$.
For every measurable subset $E\subset \Pi$ that intersects $\Pi$ $s$-leaf-wise and satisfies $\mu(E)>0$, there exists a bijective map $\theta:\big(E,\mu_E\big)\rightarrow \big(\Pi,\mu_\Pi\big)$ such that
\begin{itemize}
\item $|{\rm Jac}(\theta)(y)-1|\leq \varepsilon,~\forall y\in E$;
\item for any $y\in E$, $\theta(y)\in W^{cu}_{\loc}(y)\cap \Pi$.
\end{itemize}
\end{Proposition}
\begin{proof}
Since $\mu(E)>0$ and $E\subset\Pi$, by the SRB property of $\mu$, there exists a point $w\in E$ such that ${\rm Leb}_{W^{cu}(w)}(E)>0$. 
We now choose a measurable bijective map
$$\theta_{w}:~\Big(E\cap W^{cu}_{\rm loc}(w),~{\rm Leb}_{E,w}\Big)\overset{{\rm onto}}\longrightarrow \Big(\Pi\cap W^{cu}_{\rm loc}(w),~{\rm Leb}_{\Pi,w}\Big)$$
such that ${\rm Jac}(\theta_{w})(z)=1,~\forall z\in E\cap W^{cu}_{\rm loc}(w)$, where ${\rm Leb}_{E,w}$ and ${\rm Leb}_{\Pi,w}$ denote the normalized restrictions of  $\Leb_{W^u(w)}$ to $E$ and $\Pi$, respectively. 
This is equivalent to say that 
$$\frac{\Leb_{W^{cu}(w)}(\theta^{-1}_{w}(A))}{\Leb_{W^{cu}(w)}(E)}=\frac{\Leb_{W^{cu}(w)}(A)}{\Leb_{W^{cu}(w)}(\Pi)}\footnote{We can have this map because non-atomic Lebesgue spaces are isomorphic to each other.},~~~\forall A\subset \Pi ~~\text{measurable}.$$
Next, we define a map $\theta:E\to\Pi$. 
For each $y\in E$, let $z=h_{y,w}^s(y)$, where $h_{y,w}^s$ is the local stable holonomy map defined in Section~\ref{Sec:Pesin}.
Since $\Pi$ is a rectangle and $E$ intersects $\Pi$ $s$-leaf-wise, we have 
$$z\in W^{cu}_{\rm loc}(w) \cap W^{s}_{\rm loc}(y)\in \Pi\cap W^{s}_{\rm loc}(y) \subset E\cap W^{s}_{\rm loc}(y).$$
Then we define
$$\theta(y)=h_{w,y}^s \circ \theta_w \circ h_{y,w}^s(y).$$
By Theorem~\ref{Thm:holonomy}, if $y$ is sufficiently close to $w$, then the Jacobians of $h^{s}_{w,y}$ and $h^{s}_{y,w}$ are close to $1$.
By Corollary~\ref{Cor:SRB-U}, by choosing $\alpha>0$ sufficiently small, the density of the conditional measures of $\mu$ with respect to the normalized restrictions of  $\Leb_{W^u}$ is close to $1$.
Thus, for any given $\varepsilon>0$, by choosing $\alpha>0$ sufficiently small we have
$$|{\rm Jac}(\theta)(y)-1| \leq \varepsilon,~~\forall y\in E.$$
The fact that $\theta(y)\in W^{cu}_{\rm loc}(y)\cap \Pi$ follows directly from the definition of $\theta$.
\end{proof}

\subsubsection{Global $\varepsilon$-measure preserving map $\theta:~A\to M$}\label{SEC:gmpm}

Recall that the disjoint $(\alpha,s)$-rectangles $\{\Pi_i^s\}_{i=1}^{N_*}$ form a cover of ${\rm PES}_{\ell}$, and that $\mu(\PES_{\ell})>1-\varepsilon^2$.
Let $\Pi_*:=M\setminus \bigcup_{i=1}^{N_*}\Pi_i^s$. Then, $\mu(\Pi_*)\le\varepsilon^2$.
Let $\mathcal P$ be the partition as in Theorem~\ref{Thm:partition-VWB}. 
Since $\mu(\Pi_*)\le\varepsilon^2$, for any $N>N'\geq N_0$ one has 
\begin{equation}\label{e.remaining-set}
	\text{for}~\varepsilon\text{-almost every}~P\in \bigvee_{n=N'}^{N}T^{-n} \mathcal P,~~\mu(P\cap \Pi_*)\leq \varepsilon\mu(P).
\end{equation}

Since the measure $\mu$ has the K-property, it follows from \cite[Chapter 10.2]{CFS82} or \cite[Theorem 4.2]{LLS16} that $\mu$ is K-mixing. 
Consequently, there exists $N_0\in \NN$ such that for every $N>N'\geq N_{0}$ and for $\varepsilon$-almost every $P\in \bigvee_{n=N'}^{N}T^{-n} \mathcal P$, one has
\begin{equation}\label{e.K_mixing}
|\mu_P(\Pi_j^s)-\mu(\Pi_j^s)|=\left|\frac{\mu(P\cap \Pi_j^s)}{\mu(P)}-\mu(\Pi_j^s) \right| \leq \varepsilon \cdot \min_{1\leq i\leq N_*}\mu(\Pi_i^s)\le\varepsilon\mu(\Pi_j^s),~~\forall j=1,\cdots,N_*.
\end{equation}
By Proposition~\ref{Pro:s-leaf-wise}, enlarging $N_0$ if necessary, for every $N>N'\ge N_0$ and every $i=1,\cdots,N_*$, the following holds: for $\varepsilon$-almost every $P\in \bigvee_{n=N'}^{N}T^{-n} \mathcal P$, there exists $E_i\subset P$ intersecting $\Pi_i^s$ $s$-leaf-wise and 
\begin{equation}\label{e.E-i-P}
\mu_P(E_i)=\frac{\mu(E_i)}{\mu(P)}\ge 1-\frac{\varepsilon}{2N_* }\cdot  \min_{1 \leq i \leq N_*}\mu(\Pi_i).
\end{equation}
Note that $2\varepsilon$-almost every $P\in \bigvee_{n=N'}^{N}T^{-n} \mathcal P$ satisfies both \eqref{e.remaining-set} and \eqref{e.E-i-P}.

\bigskip

Fix such a set $P \in \bigvee_{n=N'}^{N} T^{-n} \mathcal P$.
Let $E=E_1\cap E_2\cap \cdots \cap E_{N_*}\cap (M\setminus \Pi_*)$. Then $E$ satisfies the following properties:
\begin{itemize}
\item $E\subset P$ and intersects $\Pi_i^s$ $s$-leaf-wise for every $1 \leq i \leq N_*$, since $E\cap \Pi_i^s=E_i\cap \Pi_{i}^s$;
\item $\mu(E)\ge (1-2\varepsilon)\mu(P)$. This follows from \eqref{e.remaining-set} and \eqref{e.E-i-P}:
\begin{align*}
\mu(P\setminus E)=\mu\Big(\bigcup_{i=1}^{N_*}(P\setminus E_i)\cup (P\cap\Pi_*)\Big)\le \sum_{i=1}^{N_*}\mu(P\setminus E_i)+\mu(P\cap\Pi_*)<2\varepsilon\mu(P).
\end{align*}
\item for every $1 \leq i \leq N_*$, assume that $\varepsilon<1/4$ and by the second item,
$$\mu\big((P\setminus E)\cap \Pi_i^s\big)\leq \mu(P)\cdot \sum_{i=1}^{N_{*}} \mu(P\setminus E_i) \leq \varepsilon \cdot \mu(P) \mu(\Pi_i^s)\leq 2\varepsilon\cdot \mu(E)\mu(\Pi_i^s).$$
\end{itemize}
Therefore, by K-mixing (see \eqref{e.K_mixing}), for every $1 \leq i \leq N_*$ we have
\begin{equation*}
\begin{aligned}
	\left|\frac{\mu(E\cap \Pi_i^s)}{\mu(E)}-\mu(\Pi_i^s) \right| &=\left|\frac{\mu(E\cap \Pi_i^s)-\mu(E)\mu(\Pi_i^s)}{\mu(E)}\right| \\
	&=  \left|\frac{\mu(P\cap \Pi_i^s)-\mu\big((P\setminus E)\cap \Pi_i^s\big)-\big(\mu(P)-\mu(P\setminus E)\big)\mu(\Pi_i^s)}{\mu(E)}\right|\\
	&\leq \left| \frac{\mu(P)}{\mu(E)}\right| \cdot \left(\left|\frac{\mu(P\cap \Pi_i^s)-\mu(P)\mu(\Pi_i^s)}{\mu(P)}\right| + \left| \frac{\mu(P\setminus E)\mu(\Pi_i^s)-\mu((P\setminus E)\cap \Pi_i^s)}{\mu(P)} \right|\right)\\
	&\le  \big((1-2\varepsilon)^{-1}+2+1\big)\cdot \varepsilon \cdot \mu(\Pi_i^s)\leq 6\varepsilon \cdot \mu(\Pi_i^s).
\end{aligned}
\end{equation*}
In particular, we have that $\mu(E\cap \Pi_i^s)>0$.

By Proposition~\ref{Pro:local-map}, for each $i=1,\ldots,N_*$, there exists an invertible measurable map 
$$\theta_i:(E\cap \Pi_i^s,~\mu_{E\cap \Pi_i^s}) \to (\Pi_i,\mu_{\Pi_i^s})$$
satisfying the conclusions of Proposition~\ref{Pro:local-map}.
We then define a bijection
$$\theta:(P,~\mathcal B_P,\mu_P)\to (M,\mathcal B,\mu)$$
by setting $\theta(z)=\theta_i(z)$ for each $z\in E\cap \Pi_i$.
We now show that $\theta$ satisfies \eqref{e.precise-formula} and \eqref{e.13-varpesilon}.

\begin{Lemma}\label{Lem:13varepsilon}
$\theta$ is $13\varepsilon$-measure preserving.
\end{Lemma}
\begin{proof}
For each $B\subset E$, let $B_i=B \cap \Pi_i^s$. 
By the Jacobian estimate in  Proposition~\ref{Pro:local-map}, we have
$$\left|\frac{\mu_{\Pi_i^s}(\theta_i(B_i))}{\mu_{E\cap \Pi_i}(B_i)}-1\right|\le\varepsilon$$
By the definition of $\mu_E$ and $\mu_{E\cap\Pi_i^s}$, this implies
\begin{equation}\label{eq:13e}
	\left|\frac{\mu(\theta_i(B_i))\mu(E\cap \Pi_i^s)}{\mu(B_i)\mu(\Pi_i^s)}-1 \right|\leq \varepsilon.
\end{equation}
Recall the set $E\subset P$ satisfies 
\begin{equation}\label{eq:13e2}
		\left|\frac{\mu(E\cap \Pi_i^s)}{\mu(E)}-\mu(\Pi_i^s) \right|\leq 6\varepsilon \mu(\Pi_{i}^s).
\end{equation}
Therefore, by \eqref{eq:13e}, $\theta(B_i)=\theta_i(B_i)$ and \eqref{eq:13e2} we have that
\begin{align*}
\left|\frac{\mu(\theta(B_i))\mu(E)}{\mu(B_i)}-1 \right|&\leq 
\left|\frac{\mu(\theta(B_i)) }{\mu(B_i)} \Big(\mu(E)-\frac{\mu(E\cap \Pi_i)}{\mu(\Pi_i)} \Big) \right|+
\left|\frac{\mu(\theta(B_i))\mu(E\cap \Pi_i)}{\mu(B_i)\mu(\Pi_i)}-1 \right|  \\
&\leq \left|\frac{\mu(\theta(B_i))\mu(E)}{\mu(B_i)\mu(\Pi_i)}\Big(\frac{\mu(E\cap \Pi_i^s)}{\mu(E)}-\mu(\Pi_i^s)  \Big) \right|+\varepsilon\\
&\leq \left|\frac{\mu(\theta(B_i))\mu(E)}{\mu(B_i)}\right| 6\varepsilon+\varepsilon.
\end{align*}
Hence,
$$\left|\frac{\mu(\theta(B_i))\mu(E)}{\mu(B_i)}-1 \right|\leq10\varepsilon,$$
and therefore
$$\left|\frac{\mu(\theta(B))\mu(E)}{\mu(B)}-1 \right|\leq10 \varepsilon.$$
Since $\mu(E)\geq (1-2\varepsilon)\mu(P)$, it follows that
$$\left|\frac{\mu(\theta(B))\mu(P)}{\mu(B)}-1 \right|\leq13 \varepsilon.$$
for sufficiently small $\varepsilon > 0$.
This establishes \eqref{e.13-varpesilon} and completes the proof of the lemma. 
\end{proof}

\begin{Lemma}\label{Lem:d-bar-prepare}
Let $P$ be the set defined above, and let $\theta : P \to M$ be the map introduced above.
Then, for any $n\ge 0$,
$$\mu\left(\Big\{x:\frac{1}{n}\sum_{j=0}^{n-1}\mathbb{1}_{[(T^j\mathcal P)_{x}\neq (T^j\mathcal P|P)_{(\theta^{-1}x)}]}(x)\leq \varepsilon \Big\}\right)\geq 1-\varepsilon.$$
\end{Lemma}
\begin{proof}
By the assumption on $\mathcal P$, for the given $\varepsilon > 0$, there exists $\delta > 0$ such that $\mu(B(\partial \mathcal P,\delta))<\varepsilon^2$.
Provided that $\alpha > 0$ is sufficiently small, for every $x \in \bigcup_{i=1}^{N_*} \Pi_i^s$, by the definition of $\theta$, we have $\theta^{-1}(x)\in W^{cu}_{\loc}(x)\cap \Pi_{i}^s$.
Hence, 
$$d(T^{-n}(x),T^{-n}(\theta^{-1}(x)))\leq \delta,~~\forall n\ge 0.$$
If $(T^j\mathcal P)_{x}\neq (T^j\mathcal P|P)_{(\theta^{-1}x)}$, we must have $T^{-j}(x)\in B(\partial\mathcal P,\delta)$. 
Therefore, we have
$$\mu\left(\Big\{x\in \bigcup_{i=1}^{N_*} \Pi_{i}:\frac{1}{n}\sum_{j=0}^{n-1}\mathbb{1}_{[(T^j\mathcal P)_x\neq (T^j\mathcal P|P)_{(\theta^{-1}x)}]}(x)> \varepsilon \Big\}\right)\leq \mu\left( \Big\{x: \frac{1}{n} \sum_{j=0}^{n-1}\mathbb{1}_{T^jB(\partial \mathcal P,\delta)}(x) >\varepsilon\Big\}\right).$$
By the invariance of $\mu$, one has that
\begin{align*}
\mu\left( \Big\{x: \sum_{j=0}^{n-1}\mathbb{1}_{T^jB(\partial \mathcal P,\delta)}(x) >n\varepsilon\Big\}\right)\le
\frac{1}{\varepsilon n}\int \sum_{j=0}^{n-1}\mathbb{1}_{T^jB(\partial \mathcal P,\delta)}(x){\rm d}\mu(x)\le \frac{1}{\varepsilon n} n\mu(B(\partial \mathcal P,\delta))<\varepsilon.
\end{align*}
This establishes \eqref{e.precise-formula} and completes the proof of the lemma. 
\end{proof}

\begin{proof}[Proof of Theorem~\ref{Thm:partition-VWB}]
For given $\varepsilon>0$, recall from Section \ref{SEC:gmpm} that there exists  $N_0=N_0(\varepsilon)$ such that for any $n\ge 0$ and any $N'\ge N\ge N_0$, for $\varepsilon$-almost every $P\in\bigvee_{k=N}^{N'}T^{-k}\mathcal P$ there exists a map $\theta:P\to M$ satisfying:
\begin{enumerate}
\item[(1)]  $\theta$ is $13\varepsilon$-measure preserving; and
\item[(2)]  for partitions $\{T^{i}\mathcal P\}_{i=1}^n$ in $(X,\mathcal B,\mu)$ and partitions $\{T^{i}\mathcal P|P\}_{i=1}^n$ in $(P,~\mathcal{B_P},\mu_P)$,
	$$\mu\left(\Big\{x:\frac{1}{n}\sum_{j=0}^{n-1}\mathbb{1}_{[(T^j\mathcal P)_x\neq (T^j\mathcal P|P)_{(\theta^{-1}x)}]}(x)\leq \varepsilon \Big\}\right)\geq 1-\varepsilon.$$
\end{enumerate}
By Lemma~\ref{Lem:Ornstein}, it follows that
$$ \bar d\big(\{T^{i}\mathcal P\}_{i=1}^n,\{T^{i}\mathcal P|P\}_{i=1}^n\big)<13\cdot 16\varepsilon<300\varepsilon.$$
This proves that $\mathcal P$ is a very weak Bernoulli partition for $T^{-1}$.
\end{proof}

\appendix

\section{Equivalent descriptions of weakly mixing}

A measurable subset $E \subset \mathbb{R}^+$ is said to have (asymptotic) density zero if 
$$\lim_{t\to+\infty}\frac{1}{t}{\rm Leb}(E\cap[0,t])=0.$$
\begin{Proposition}\label{Pro:density-zero}
A measure-preserving flow $(X,\mathcal B,\mu,\{T^t\})$ is weakly mixing if and only if for any measurable sets $A,B$, there is a subset $E\subset\mathbb R^+$ of density zero such that
$$\lim_{E\not\ni t \to \infty} \mu(T^{-t}(A) \cap B) = \mu(A) \cdot \mu(B).$$
\end{Proposition}

\begin{Corollary}\label{Cor:manysets-intersect}
Assume $(X,\mathcal B,\mu,\{T^t\})$ is weakly mixing. 
For any measurable sets $B,U_1,U_2,\cdots,U_N$ with positive measures, there exist $\delta>0$ and a sequence $t_n\to\infty$ such that
$$\mu(T^{t_n}B\cap U_i)>\delta>0,~~\forall n>0,~~\forall 1\leq i\leq N.$$
\end{Corollary}
\begin{proof}
Take $0<\delta<\min_{1\le i\le N} \mu(B)\mu(U_i)$. 
For each $1\leq i\leq N$, by Proposition~\ref{Pro:density-zero}, there exists a set $E_i\subset \mathbb{R}^+$ of density $0$  such that $  \mu(T^{-t}(U_i) \cap B) \to \mu(B)\mu(U_i)$ as $E_i\not\ni t\to +\infty$.
Let $E=\bigcup_{1\le i\le N}E_i$. 
Then, $E$ has density $0$, and $\mu(T^{t_n}B\cap U_i)>\delta$ for all sufficiently large $t\notin E$.
\end{proof}
The proof of Proposition~\ref{Pro:density-zero} is based on the following lemma concerning measurable functions.

\begin{Lemma}\label{Lemma:measurable-function}
Let $f:~\mathbb R^+\to \mathbb R$ be a bounded measurable function. 
Then 
$$\lim_{t\to\infty}\frac{1}{t}\int_0^t|f(s)|{\rm d}s=0$$
if and only if there is a set $E\subset\mathbb R^+$ of density $0$ such that $~\lim_{E\not\ni t \to \infty} f(t)=0$.
\end{Lemma}
\begin{proof}
Assume that there exists a set $E\subset \mathbb{R}^+$ of density zero such that $\lim_{E\not\ni t \to \infty} f(t)=0$. 
This means that for any $\varepsilon>0$, there is $T>0$ such that $|f(t)|<\varepsilon$ for any $E\not\ni t>T$. 
Note that 
$$\lim_{t\to\infty}\frac{1}{t}{\rm Leb}(E\cap[T,t])=0.$$
Thus,
\begin{align*}
	\lim_{t \to \infty} \frac{1}{t} \int_0^t |f(s)|{\rm d}s&=\lim_{t \to \infty} \frac{1}{t} \int_{T}^t |f(s)|{\rm d}s \\
	&\leq  \limsup_{t \to \infty} \frac{1}{t} \int_{[T,t] \cap E} |f(s)|{\rm d}s + \limsup_{t \to \infty} \frac{1}{t} \int_{[T,t] \setminus E} |f(s)|{\rm d}s \\
	&\leq   \sup_{s \in \mathbb{R}^+} |f(s)| \cdot \limsup_{t \to \infty} \frac{1}{t} \Leb(E \cap [T,t])+\varepsilon \\
	&\leq \varepsilon.
\end{align*}
The conclusion follows from the arbitrariness of $\varepsilon$.

Assume that there does not exist a set $E\subset\mathbb R^+$ of density $0$ such that $\lim_{E\not\ni t \to \infty} f(t)=0$. 
Then, there exist $\varepsilon_0>0$ and $E_0\subset\mathbb R$ such that $\limsup\limits_{t\to\infty}\Leb(E_0\cap[0,t])/t\ge\varepsilon_0$, and $|f(t)|\ge\varepsilon_0$ for any $t\in E_0$. Thus,
$$
\limsup_{t \to \infty} \frac{1}{t} \int_0^t |f(s)|{\rm d}s \geq \limsup_{t \to \infty} \frac{1}{t} \int_{E_0 \cap [0,t]} |f(s)| {\rm d}s
\geq \limsup_{t \to \infty} \frac{\varepsilon_0}{t} \operatorname{Leb}(E_0 \cap [0,t]) \geq \varepsilon_0^2>0.
$$
which is a contraction.
\end{proof}

\begin{proof}[Proof of Proposition~\ref{Pro:density-zero}]
Given two measurable sets $A$ and $B$, define
$$f(t):=\mu(A \cap T^{-t}(B)) -\mu(A)\mu(B).$$ 
Then, $f:\mathbb{R}^+ \to \mathbb{R}$ is a bounded measurable function.
It therefore suffices to apply Lemma~\ref{Lemma:measurable-function} to conclude.
\end{proof}

\begin{Definition}\label{Def:eigenfunction}
Given a measure-preserving flow $(X,\mathcal B,\mu,\{T^t\})$, a complex-valued nonzero function $f\in L^2(\mu)\setminus\{0\}$ is called an eigenfunction of the the flow $\{T^t\}$ if there is $\omega\in\mathbb R$ such that
$$f\circ T^t={\rm e}^{2\pi i\omega t}f,~~~\forall t\in\mathbb R.$$
The number $\omega$ is said to be eigenfrequency, and ${\rm e}^{2\pi i\omega}$ is said to be eigenvalue.
\end{Definition}
By Definition~\ref{Def:eigenfunction}, the constant functions are eigenfunctions corresponding to the eigenvalue $1$ (equivalently, eigenfrequency $0$).

\begin{Remark}
If $\mu$ is ergodic, then by Birkhoff ergodic theorem, eigenfunctions associated to the eigenvalue $1$ must be constant functions. 
Ergodic systems may have other eigenvalues.
\end{Remark}

\begin{Proposition}\label{Pro:eigenfunction}
A measure-preserving flow $(X,\mathcal B,\mu,\{T^t\})$ is weakly mixing, then its eigenfunctions are constant functions, hence it only has $1$ as its eigenvalue.
More precisely, if $f\in L^2(\mu)$ and $f\circ T^t={\rm e}^{2\pi i\omega t}f$ for any $t\in\mathbb R$, then $f$ is constant and $\omega=0$.
\end{Proposition}

\begin{proof}
Note that a weakly mixing system must be ergodic (see \cite[Remark 3.4.4]{FiH19}, for instance). 
Thus, when $\omega=0$, we must have that $f$ is constant.
We prove by contradiction that there is $\omega\neq 0$ which is an eigenfrequency of the flow. 
Since
$$\forall t\in\mathbb R,~~~\langle f, 1 \rangle = \int f {\rm d}\mu= \int f \circ \varphi^t{\rm d}\mu = \int {\rm e}^{2\pi i\omega t} f {\rm d}\mu= {\rm e}^{2\pi i\omega t} \langle f, 1 \rangle,$$
one has that $\langle f, 1 \rangle=0$, together with the fact that $\omega\neq 0$. Thus, for every $T>0$
\begin{align*}
\int |f|^2{\rm d}\mu&=\frac{1}{T}\int_0^T |e^{2\pi it\omega}|\big(\int f \cdot \bar{f}{\rm d}\mu\big){\rm d}t \\
&=\frac{1}{T}\int_0^T \left|\int e^{it\omega}f\bar{f}{\rm d}\mu\right|{\rm d}t \\
&=\frac{1}{T}\int_0^T \left|\int (f\circ\varphi^t)\bar{f}{\rm d}\mu\right|{\rm d}t
\end{align*}
Since $\langle f, 1 \rangle=0$ and  $\mu$ is weakly mixing, we have $T^{-1}\int |\int (f\circ\varphi^t)\bar{f}{\rm d}\mu|{\rm d}t\to 0$ as $ T\to +\infty$.
This implies that $\int |f|^2{\rm d}\mu=0$ and so $f=0$, which is a contradiction.
\end{proof}
%
%
%
%
%
%
%

\section*{Acknowledgments}
We thank O. Sarig, who explained many details of Bernoulli systems to us; we thank J. Buzzi and Y. Lima for much information on Bernoulli systems. 
We also thank the Tianyuan Mathematical Center in Northeast China for its support.

\vskip 5pt

\flushleft{\bf Chiyi Luo} \\
\small School of Mathematics and Statistics, 
Jiangxi Normal University, Nanchang,   330022, P. R. China\\
\textit{E-mail:} \texttt{luochiyi98@gmail.com}\\
\flushleft{\bf Dawei Yang} \\
\small School of Mathematical Sciences,  Soochow University, Suzhou, 215006, P.R. China\\
\textit{E-mail:} \texttt{yangdw@suda.edu.cn}\\
\end{document}